\documentclass[letterpaper,11pt,reqno]{amsart} 
\RequirePackage[OT1]{fontenc}
\usepackage[portrait,margin=3cm]{geometry}  
\usepackage{mathrsfs}
\usepackage[colorlinks,citecolor=blue,urlcolor=blue,linkcolor=blue,pagecolor=blue,linktocpage=true,backref=true]%
{hyperref}
\usepackage[foot]{amsaddr}
\usepackage{amssymb,amsthm,amsfonts,amsbsy,latexsym,dsfont}
\usepackage{graphicx}
\usepackage[numeric,initials,nobysame]{amsrefs}
\usepackage{upref,setspace}

\usepackage{enumerate}
\newenvironment{enumeratei}{\begin{enumerate}[\upshape (i)]}{\end{enumerate}}
\newenvironment{enumeratea}{\begin{enumerate}[\upshape (a)]}{\end{enumerate}}
\newenvironment{enumeraten}{\begin{enumerate}[\upshape 1.]}{\end{enumerate}}

\numberwithin{equation}{section}
\numberwithin{figure}{section}
\numberwithin{table}{section}

\usepackage{color}              
\usepackage{graphicx}
\usepackage[]{amsmath}
\usepackage{amssymb}
\usepackage{hyperref}
\usepackage{amsfonts}
\definecolor{MyDarkBlue}{rgb}{0,0.08,0.50}
\definecolor{BrickRed}{rgb}{0.65,0.08,0}
\usepackage[textsize=small]{todonotes}       
\usepackage{amsmath, amsthm, amssymb}

\newcommand{\bfw}{ {\mathbf{w}}}
\newcommand{\mirror}{ \Gamma }

\newcommand{\BS}{{\bf{BSR}}}

\newcommand{\BF}{{\bf {BF}}}

\newcommand{\BM}{{\mbox{BM}}}

\newcommand{\DD}{\mathcal{D}}
\newcommand{\CC}{\mathcal{C}}

\newcommand{\ompar}{\varpi}

\newcommand{\vol}{\mbox{vol}}
\newcommand{\Imm}{\mbox{Imm}}

\newcommand{\clc}{\mathcal{C}}

\newcommand{\clv}{\mathcal{V}}
\newcommand{\clg}{\mathcal{G}}

\newcommand{\clb}{\mathcal{B}}

\newcommand{\KK}{\mathcal{K}}
\newcommand{\VV}{\mathcal{V}}
\newcommand{\NNN}{\mathbb{N}}

\newcommand{\RRR}{\mathbb{R}}

\newcommand{\cla}{\mathcal{A}}

\newcommand{\RG}{{\bf{RG}}}

\newcommand{\XX}{\mathcal{X}}
\newcommand{\bfx}{{\bf{x}}}
\newcommand{\bfy}{{\bf{y}}}
\newcommand{\bfv}{{\bf{v}}}
\newcommand{\bfG}{{\bf{G}}}
\newcommand{\bfg}{{\bf{g}}}

\newcommand{\be}{\begin{equation}}
\newcommand{\ee}{\end{equation}}
\newcommand{\beq}{\begin{eqnarray*}}
\newcommand{\eeq}{\end{eqnarray*}}

\newcommand{\beqn}{\begin{eqnarray}}
\newcommand{\eeqn}{\end{eqnarray}}

\newcommand{\ba}{\begin{aligned}}
\newcommand{\ea}{\end{aligned}}
\newcommand{\bes}{\begin{equation*}}
\newcommand{\ees}{\end{equation*}}

\newtheorem{Lemma}{Lemma}[section]
\newtheorem{Proposition}[Lemma]{Proposition}

\newtheorem{Theorem}[Lemma]{Theorem}

\theoremstyle{definition}

\newtheorem{Corollary}[Lemma]{Corollary}

\theoremstyle{definition}

\newcommand{\EE}{\mathcal{E}}

\newcommand{\prob}{\mathbb{P}}
\newcommand{\PP}{\mathcal{P}}
\newcommand{\E}{\mathbb{E}}

\newcommand{\calA}{\mathcal{A}}

\newcommand{\FF}{\mathcal{F}}

\newcommand{\WW}{\mathcal{W}}
\newcommand{\eps}{\varepsilon}

\newcommand{\set}[1]{\left\{#1\right\}}

\newcommand{\Rbold}{{\mathbb{R}}}

\newcommand{\ind}[2]{1_{(e \in \pi(#1,#2))}}

\newcommand{\bfd}{{\bf d}}
\newcommand{\bfzero}{{\bf 0}}
\def\ind{{\rm 1\hspace{-0.90ex}1}}

\newcommand{\calS}{\mathcal{S}}

\newcommand{\sss}{\scriptscriptstyle}

\newcommand{\barx}{\bar{x}}
\def\1{{\mathchoice {1\mskip-4mu\mathrm l}      
{1\mskip-4mu\mathrm l}
{1\mskip-4.5mu\mathrm l} {1\mskip-5mu\mathrm l}}}

\newcommand {\convd}{\stackrel{d}{\longrightarrow}}
\newcommand {\convp}{\stackrel{\sss {\mathbb P}}{\longrightarrow}}

\newcommand{\erdos}{Erd\H{o}s-R\'enyi }

\numberwithin{equation}{section}

\newcommand{\cB}{\mathcal{B}}
\newcommand{\cD}{\mathcal{D}}
\newcommand{\cI}{\mathcal{I}}
\newcommand{\cK}{\mathcal{K}}

\newcommand{\cQ}{\mathcal{Q}}

\newcommand{\cW}{\mathcal{W}}\newcommand{\cX}{\mathcal{X}}
  
\newcommand{\vM}{\mathbf{M}}

\newcommand{\va}{\mathbf{a}}\newcommand{\vc}{\mathbf{c}}
\newcommand{\vd}{\mathbf{d}}

\newcommand{\mvalpha}{\boldsymbol{\alpha}}
\newcommand{\mvgamma}{\boldsymbol{\gamma}}

\newcommand{\fI}{\mathfrak{I}}

\begin{document}
	
	\title[Subcritical bounded-size rules]{Bounded-size rules: The barely subcritical regime}

	\date{\today}

	\author[Bhamidi]{Shankar Bhamidi}
	\address{Department of Statistics and Operations Research, 304 Hanes Hall CB \#3260, University of North Carolina, Chapel Hill, NC 27599}
	\author[Budhiraja]{Amarjit Budhiraja}
	\author[Wang]{Xuan Wang}
	\email{bhamidi@email.unc.edu, budhiraj@email.unc.edu, wangxuan@email.unc.edu}

\begin{abstract}
	Bounded-size rules(BSR) are dynamic random graph processes which incorporate limited choice along with randomness in the evolution of the system. Typically one starts with the empty graph and at each stage  two edges are chosen uniformly at random.  One of the two edges  is then placed into the system according to a decision rule   based on the sizes of the components containing the four vertices.
 For bounded-size rules, all components of size greater than some fixed $K\geq 1$ are accorded the same treatment. Writing $\BS(t)$ for the state of the system with $\lfloor nt/2\rfloor$ edges, Spencer and Wormald \cite{spencer2007birth} proved that for such rules, there exists a (rule dependent) critical time $t_c$ such that when $t< t_c$ the size of the largest component is of order $\log{n}$ while for $t> t_c$, the size of the largest component is of order $n$.
In this work we obtain upper bounds (that hold with high probability) of order $n^{2\gamma} \log ^4 n$, on the size of the largest component, at time instants $t_n = t_c-n^{-\gamma}$, where 
 $\gamma\in (0,1/4)$. This result for the barely subcritical regime forms a key ingredient in the study undertaken in \cite{amc-2012}, of the asymptotic dynamic behavior of the process describing the vector
of component sizes and associated complexity of the components for such random graph models in the critical scaling window.
The proof uses a coupling of  BSR processes with a certain 
family of inhomogeneous random graphs with vertices in the type space $\Rbold_+\times \cD([0,\infty):\NNN_0)$ where $\cD([0,\infty):\NNN_0)$ is the Skorohod $D$-space of functions that are right continuous and have left limits, with values in the space of nonnegative integers $\NNN_0$,  equipped with the usual Skorohod topology. The coupling construction also gives an alternative characterization (than the usual explosion time of the susceptibility function) of the critical time $t_c$ for the emergence of the giant component in terms of the operator norm of integral operators
on certain $L^2$ spaces. 
\end{abstract}

\maketitle

\section{Introduction}
\label{sec:intro}
The classical \erdos random graph can be thought of as a dynamic random graph process on the vertex set $[n]:=\set{1,2,\ldots, n}$, where one starts with the empty graph $\bfzero_n$ (the graph with $n$ vertices and no edges) and at each discrete time step  chooses an edge uniformly at random and places it in the configuration. Denote by ${\bf ER}^{\sss (n)}(t)$ for the state of the graph obtained after $\lfloor nt/2 \rfloor$ steps. For any graph $\bfG$, denote $\CC_i(\bfG)$ for the $i$-th largest component, and $|\CC_i(\bfG)|$ for its size (number of vertices).
Classical results (\cite{er-1,er-2,bollobas-rg-book}) say that for fixed $t< 1$, the size of the largest component $|\CC_1( {\bf ER}^{\sss (n)}(t))|$ is $ O(\log{n})$ while for $t> 1$, $|\CC_1( {\bf ER}^{\sss (n)}(t))| \sim f(t) n$. 
Here $f(t) > 0$ is the survival probability of an associated supercritical branching process. For $t> 1$, the size of the second largest component $|\CC_2( {\bf ER}^{\sss (n)}(t))| = O(\log{n})$. The largest component is often referred to as as the giant component. 

There have been several works aimed at  understanding the nature and emergence of this giant as $t$ transitions from below to above $t_c =1$ (\cite{aldous1997brownian,janson1994birth}). In recent years, motivated by a question of Achlioptas, there has been
a significant interest in investigating more general dynamical random graph models.  The driving theme has been to understand the role of limited choice along with randomness in the evolution of the network, in particular the time and nature of emergence of the giant component. 
The simplest such model that has been  rigorously analyzed,  referred to as the Bohman-Frieze process, can be described as follows. Start with the empty graph at time $t=0$. At each discrete time step, choose two edges $e_1, e_2$ uniformly at random amongst all pairs of ordered edges. Place edge $e_1 = (v_1, v_2)$ if both end points $v_1, v_2$ are isolated vertices (components of size one), otherwise use edge $e_2$.   

Despite this conceptually simple modification of the standard \erdos random graph process, a rigorous understanding of this process turns out to be  non-trivial. Write ${\bf BF}^{\sss (n)}(t)$ for the state of the system when we have placed $\lfloor nt/2 \rfloor$ edges.  Bohman and Frieze (\cite{bohman2001avoiding,bohman2004avoidance}) showed that there exists a time $t> 1$ and $\eps \in (0,1)$ such that the size of the largest component $|\CC_1({\bf BF}^{\sss (n)}(t))| = o(n^\eps)$. Thus this simple modification delays the time of emergence of a giant component.  

In \cite{spencer2007birth}, these results were substantially refined and  extended  to the context of all \emph{bounded-size rules} (BSR) which we now describe. 

\textbf{The bounded-size $K$-rule process $\{\BS^{\sss(n)}(t)\}_{t\ge 0}$.}
 Fix $K\geq 0$, this will be a parameter in the construction of the process. Bounded-size rules treat all components of size greater than $K$ in an identical fashion.  Let $\Omega_K =\set{1,2,\ldots, K, \ompar}$. Conceptually $\ompar$ represents components of size greater than $K$.  Given a graph $\bfG$ and a vertex $v\in\bfG$, write $\CC_v(\bfG)$ for the component that contains $v$. Define 
\begin{equation}
\label{eqn:cv-def}
	c_{\bfG}(v) := \left\{\begin{array}{ll}
		|\CC_v(\bfG)| & \mbox{ if }|\CC_v(\bfG)|\leq K\\
		\ompar & \mbox{ if } |\CC_v(\bfG)| > K
	\end{array}\right.
\end{equation}
For a quadruple of (not necessarily distinct) vertices $v_1, v_2, v_3, v_4$, write $\vec{v}$ for the ordered quadruple $\vec{v} = (v_1,v_2,v_3,v_4)$. Let $c_{\bfG}(\vec{v}) = (c_{\bfG}(v_1), c_{\bfG}(v_2), c_{\bfG}(v_3), c_{\bfG}(v_4))$. Fix $F\subseteq \Omega_K^4$. 
The set $F$ will be another parameter in the construction of the process.
The $F$-bounded-size rule($F$-BSR) is defined as follows:
\begin{enumeratea}
	\item At time $k=0$ start with the empty graph $\BS_0^{\sss(n)}:=\bfzero_n$ on $[n]$ vertices. 
	\item For $k\geq 0$, having constructed the graph $\BS_k^{\sss (n)}$,  construct $\BS_{k+1}^{\sss(n)}$ as follows. Choose $4$ vertices $\vec{v} = (v_1, v_2, v_3, v_4)$ uniformly at random amongst all $n^4$ possible quadruples and let 
	$c_k(\vec{v}) = c_{\BS_k}(\vec{v})$. If $c_k(\vec{v}) \in F$ then $\BS_{k+1}^{\sss(n)} = \BS_k^{\sss(n)}\cup (v_1, v_2) $ else $\BS_{k+1}^{\sss(n)} = \BS^{\sss(n)}_k \cup (v_3, v_4)$. 
\end{enumeratea}  
We are interested in the dynamics of the rescaled process $\BS^{\sss(n)}(t) = \BS_{\lfloor nt/2 \rfloor}^{\sss(n)} $. Mathematically it is more convenient  to work with a formulation in which edges are added at Poissonian time instants rather than at fixed discrete 
times.  More precisely, we will consider the following random graph process (denoted once more as $\BS^{\sss(n)}(t)$).
For every quadruple of vertices $\vec{v} = (v_1, v_2, v_3, v_4)\in [n]^4$, let $\PP_{\vec{v}}$ be a Poisson process with rate ${1}/{2n^3}$, independent between quadruples. Note that this implies that the rate of creation of edges is $n^4 \times 1/2n^3 = n/2$. Thus we have sped up time by a factor $n/2$ as in the above  discrete time construction. Start with $\BS^{\sss(n)}(0) =\bfzero_n$. For any $t\geq 0$, at which there is a point in $\PP_{\vec{v}}$ 
for a quadruple  $\vec{v} \in [n]^4$, define
\begin{equation}
	\label{eqn:f-rule-def-cts}
 	\BS^{\sss(n)}(t) =\left\{ \begin{array}{ll}
 		\BS^{\sss(n)}(t-)\cup (v_1, v_2) & \quad \mbox{ if } c_{t-}(\vec{v}) \in F\\
\BS^{\sss(n)}(t-)\cup (v_3, v_4) & \quad \mbox{ otherwise, }
 	\end{array} \right. 
 \end{equation}
where $c_{t-}(\vec{v}) = c_{\BS(t-)}(\vec{v})$.  

Two examples of such processes are  
 Erd\H{o}s-R\'{e}nyi process (here  $K=0$, $\Omega_K=\{\ompar\}$ and $ F = \set{ (\ompar, \ompar, \ompar,\ompar) }$) and 
 Bohman-Frieze process (here $K=1$, $\Omega_K=\{1,\ompar\}$ and $ F =\{ (1,1, j_3, j_4): j_3,j_4 \in \Omega_K \}$).
Spencer and Wormald\cite{spencer2007birth} showed that  every bounded-size rules exhibits a phase transition similar to the \erdos random graph process.  More precisely, write $\CC_i^{\sss(n)}(t)$ for the $i$-th largest component in $\BS^{\sss(n)}(t)$, and $|\CC_i^{\sss(n)}(t)|$ for the size of this component.  Define the \emph{susceptibility} function  
\begin{equation}
\label{eqn:suscept-def}
	\calS_2(t) = \sum_{i=1}^\infty |\CC_i^{\sss (n)}(t)|^2.
\end{equation}  
Then \cite{spencer2007birth} proves the following result. 
\begin{Theorem}[Theorem 1.1, \cite{spencer2007birth}]\label{thm:spencer-wormald}
Fix $F \in \Omega_K^4$.  Then for the random graph process associated with the $F$-BSR, there exists deterministic monotonically increasing function $s_2(t)$ and a critical time $t_c$ such that $\lim_{t \uparrow t_c} s_2(t) = \infty$ and
	$$ \frac{\calS_2(t)}{n} \convp s_2(t) \mbox{ as } n \to \infty \mbox{, for all } t \in [0,t_c). $$
For fixed $t< t_c$, $|\CC_1^{\sss(n)}(t)|= O(\log{n})$ while for $t> t_c$, $|\CC_1^{\sss(n)}(t)| = \Theta_P(n)$. 		
\end{Theorem}
Here we use $o, O,\Theta$ in the usual manner.  Given a sequence of random variables $\{\xi_n\}_{n\ge 1}$ and a function $f(n)$, we say $\xi_n = O (f)$ if there is a constant $C$ such that $\xi_n \le C f(n)$ with high probability (whp), and we say $\xi_n = \Omega(f)$ if there is a constant $C$ such that $\xi_n \ge Cf(n)$ whp. Say that $\xi_n = \Theta(f)$ if $\xi_n = O(f)$ {\bf and} $\xi_n = \Omega(f)$. 
In addition, we say $\xi_n = o (f)$ if $\xi_n/f(n) \convp 0$.

Thus as $t$ transitions from less than $t_c$ to greater than $t_c$, the size of the largest component jumps from size $O(\log{n})$ to a giant component $\Theta(n)$. The aim of this work is to study the barely subcritical regime, i.e.
to analyze the behavior of the size of the largest component
at times $t= t_c-\eps_n$ where $\eps_n\to 0$. The following is the main result. 
\begin{Theorem}[{\bf Barely subcritical regime}]
\label{thm:subcrit-reg}
Fix  $F \subset \Omega_K^4$ and $\gamma \in (0,1/4)$. Then there exists  $B\in  (0, \infty)$ such that,
\[ \prob\set{ |\CC_1^{\sss (n)}(t)| \le B\frac{(\log n)^4}{(t_c-t)^2},~ \forall t \le t_c-n^{-\gamma} } \to 1,  \]
as $n \to \infty$. 
\end{Theorem}
As another consequence of our proofs, we obtain  an alternative characterization of the critical time for a bounded-size rule given in Theorem \ref{theo:new-crit-time} below. Let 
$\cX = [0,\infty)\times \cD([0,\infty):\NNN_0)$
where $\cD([0,\infty):\NNN_0)$ is the Skorohod $D$-space of functions that are right continuous and have left limits with values in the space of nonnegative integers,  equipped with the usual Skorohod topology. 
Given  a finite measure $\mu$ on $(\cX, \cB(\cX))$ and a measurable map $\kappa: \cX \times \cX \to [0, \infty)$ satisfying $\int_{\cX \times \cX}\kappa^2(\bfx,\bfy)\mu(d\bfx)\mu(d\bfy) < \infty$, define the integral operator
$\cK: L^2(\cX, \mu) \to L^2(\cX, \mu)$ as
$$\cK f(\bfx) = \int_{\cX} \kappa(\bfx,\bfy) f(\bfy) \mu(d\bfy), f \in L^2(\cX, \mu), \; \bfx \in \cX .$$
We refer to $\kappa$ as a kernel on $\cX \times \cX$ and $\cK$ as the integral operator associated with $(\kappa, \mu)$.
We will show the following result.   
\begin{Theorem}[{\bf Characterization of the critical time}]
	\label{theo:new-crit-time}
	Fix  $F \subset \Omega_K^4$. Then there exists a collection of  kernels $\set{\kappa_t}_{t\geq 0}$ on $\cX \times \cX$ and finite measures $\set{\mu_t}_{t\geq 0}$ on $(\cX, \cB(\cX))$ such that the  integral operators $\cK_t$ associated with
	$(\kappa_t, \mu_t)$, $t > 0$,
	have the property that the operator norms $\rho(t) = \|\cK_t\|$ are continuous and strictly increasing in $t$.  Furthermore, $t_c$ is the unique time instant such that $\rho(t_c) = 1$.
	%
	%
\end{Theorem}
See Section \ref{sec:other-random-graph-processes} for a precise definition of $\kappa_t$ and $\mu_t$.  We postpone the discussion of the connection between the integral operators in Theorem \ref{theo:new-crit-time}
and the BSR processes to Section \ref{sec:discussion}. Using arguments similar to \cite{bf-spencer-perkins-kang}  for the Bohman-Frieze model, one can show that for any fixed $\eps> 0$, the size of the largest component 
at time $t=t_c-\eps$ can be lower bounded as $|\CC_1^{\sss(n)}(t)|\geq  A \log{n}/(t_c-t)^2$ where $A$ is a constant independent of $\eps$. Thus up to logarithmic factors one expects the upper bound in Theorem \ref{thm:subcrit-reg} to be tight.


Theorem \ref{thm:subcrit-reg} plays a central role in  the study of the asymptotic dynamic behavior of the process describing the vector
of component sizes and associated surpluses for BSR processes in the critical scaling window and its connections with the augmented multiplicative coalescent process.  This study is the subject of  \cite{amc-2012} to which we refer the reader for details.


\subsection{Organization of the paper}
The paper is organized    as follows. In Section \ref{sec:discussion} we give a discussion  of the main result. Section \ref{sec:not} collects some notation used in this work. In Section \ref{sec:bsr-diff-eqns} we introduce and analyze
 certain inhomogeneous random graph processes associated with the BSR process.  Finally, in Section \ref{sec:bsr-largest-component} we  complete the proofs of Theorems \ref{thm:subcrit-reg}  and \ref{theo:new-crit-time}.

\section{Discussion}
\label{sec:discussion}
We now give some background, open problems and general discussion of the results in this work.

 \subsection{Subcritical and supercritical random graphs:} There has been a considerable interest  in understanding various properties of random graph models in the barely subcritical and supercritical regime, see for example (\cite{pittel-subcrit,janson-subcrit,janson-luczak-subcrit}) for various results on complex network models such as the configuration model in the subcritical regime, \cite{peres-ding,poisson-cloning} for structural properties of such graphs including mixing times of the nearly supercritical \erdos random graph, \cite{van2012hypercube} for an analysis of the Hamming cube near the critical regime and \cite{bollobas-riordan-janson} for an extensive analysis of a general class of the inhomogeneous random graphs.  
In recent  years there has been a significant effort in understanding a special (non-bounded-size) rule called the Product rule (\cite{achlioptas2009explosive}), where one uses the edge that minimizes the product of the components at the end points. Simulations in \cite{achlioptas2009explosive}  suggest that the nature of the emergence of the giant component is different from that observed for rules such as the \erdos or Bohman-Frieze process. Conceptually such rules tend to be harder to analyze since one needs to keep track not just of vertices in components upto size $K$ for $K< \infty$ but for all $K$. There has been recent progress in understanding such rules \cite{riordan2011achlioptas}. The subcritical regime for such processes has been studied in \cite{riordan2012evolution} where it is shown that there exists a critical time $t_c$ such that for $t> t_c$, the susceptibility function defined as in \eqref{eqn:suscept-def} satisfies $\calS_2(t)/n\convp \infty$ as $n\to\infty$. Furthermore there exist functions $\set{f_k(\cdot)}_{k\geq 1}$ such that for $t< t_c$ the proportion of vertices in components of size $k$ remains closely concentrated about $f_k(t)$ as $n\to\infty$.

For bounded-size rules, the only known results in the barely subcritical regime are in the context of the Bohman-Frieze process in \cite{bf-spencer-perkins-kang} and \cite{bhamidi-budhiraja-wang2011} (also see \cite{janson2010phase} where scaling exponents for susceptibility functions in the Bohman-Frieze process were derived). In \cite{bf-spencer-perkins-kang} it was shown that for fixed $\eps$, the largest component in the Bohman-Frieze process at time $t_c-\eps$ satisfies  $|\CC_1(\BF^{\sss (n)}(t_c-\eps))| = \Theta (\log{n}/\eps^2)$. In \cite{bhamidi-budhiraja-wang2011}, upper bounds when $\eps = \eps_n$, 
where $\eps_n \to 0$, were obtained and  a result  analogous to Theorem \ref{thm:subcrit-reg} was shown using the special structure of certain differential equations associated with the Bohman-Frieze process. 
  
\subsection{Optimal scaling and open questions:}  
Note that for a fixed $\eps > 0$, Theorem \ref{thm:subcrit-reg} gives an upper bound of $B(\log{n})^4/\eps^2$. One would expect, similar to the special case of the Bohman-Frieze process treated in \cite{bf-spencer-perkins-kang}, that for fixed $\eps$, $|\CC_1(\BF^{\sss (n)}(t_c-\eps))|= \Theta(\log{n}/\eps^2)$. It would be interesting to see if the results in the paper can be refined to prove this result for general bounded-size rules. 
Also note that Theorem  \ref{thm:subcrit-reg} considers time instants 
 $t_n = t_c -n^{-\gamma}$ for $\gamma \in (0,1/4)$. A separate analysis that uses this result as a starting point is carried out in \cite{amc-2012} to treat the 
  component sizes in the critical scaling window i.e. when $\gamma =1/3$. It would be interesting to extend the analysis of the current paper
to get  refined estimates on the size of the largest component
when $ \gamma \in [1/4,1/3)$.

\subsection{Connection to the discrete time system:} 
We use the continuous time construction given in terms of Poisson processes, as opposed to the discrete time construction, for mathematical convenience (see e.g. the various martingale estimates in Section \ref{sec:bsr-diff-eqns}). 
It is easy to show that in the continuous time construction, by time $t_n = t_c - n^{-\gamma}$, the number of edges in the system is of the order $nt_c - n^{1-\gamma} + O(\sqrt{n}) $. 
Using this and the monotonicity of the process it is easy to check that Theorem \ref{thm:subcrit-reg} holds for the discrete time version as well.

\section{Notation}
\label{sec:not}
We collect some notation used through the rest of the paper. 
All unspecified limits are taken as $n \to \infty$. We use $\convp$ and $\convd$ to denote convergence in probability and in distribution respectively.
Given a sequence of events $\{E_n\}_{n\ge 1}$, we say $E_n$  occurs with high probability (whp) if $\prob\{E_n\} \to 1$.


For a set $S$ and a function $\bfg: S \to \RRR^k$, we write $||\bfg||_{\infty} = \sum_{i=1}^k\sup_{s \in S}|g_i(s)|$, where $\bfg = (g_1, \cdots g_k)$.
For a Polish space $S$, we denote by $\BM(S)$, the space of bounded measurable functions on $S$ (equipped with the Borel sigma-field $\clb(S)$.)
For a finite set $S$, $|S|$ denotes the number of elements in the set. $\NNN_0$ is the set of nonnegative integers.
For ease of notation, we shall often suppress the dependence on $n$ and shall write for example $\BS(t) = \BS^{\sss(n)}(t)$.
Recall the Poisson  processes $\PP_{\vec{v}}$ used to construct $\BS(\cdot)$ in Section \ref{sec:intro}. Let $\set{\FF_t}_{t\geq 0}$ be the associated  filtration: $\FF_t = \sigma\set{\PP_{\vec{v}}(s): s\leq t, \vec{v}\in [n]^4}$. We shall often deal with $\set{\FF_t}$-semimartingales $\set{J(t)}_{t\geq 0}$  of the form 
\begin{equation}
\label{eqn:semi-mart-def}
dJ(t):= \alpha(t) dt + dM(t),
\end{equation} 
where $M$ is a $\set{\FF_t}$ local martingale. We shall denote $\alpha = \vd(J)$ and $M=\vM(J)$.  For a local martingale $M(t)$, we shall write $\langle M, M\rangle(t)$ for the predictable quadratic variation process namely the predictable process of bounded variation such that $M(t)^2- \langle M,M \rangle(t)$ is a local martingale.

\section{Inhomogeneous random graphs}
\label{sec:bsr-diff-eqns}

Fix $K\geq 0$ and a general bounded-size rule $F\subseteq \Omega_K^4$ and recall that $\set{\BS(t)}_{t\geq 0}$ denotes the continuous time bounded-size rule process started with the empty graph at $t=0$. Note that $K=0$ case corresponds to the \erdos random graph process for which results such as Theorem \ref{thm:subcrit-reg} are well known. Thus, henceforth we shall assume $K\geq 1$.  We begin in Section \ref{sec:small-component}   by analyzing the proportion of vertices in components of size $i$ for $i\leq K$. As shown in \cite{spencer2007birth}, these converge to a set of deterministic functions which 
can be characterized as the unique solution of a set of differential equations. We will need precise rates of convergence for these proportions  which we establish in  Lemma \ref{lemma:approx-xi}. We then study the evolution of components of size larger than $K$ in Section \ref{sec:evol-bst-st}. Finally, we relate the evolution of these components to an inhomogeneous random graph (IRG) model in  Section \ref{sec:other-random-graph-processes}.

\subsection{Density of vertices in components of size bounded by $K$}
\label{sec:small-component}
Recall from Section \ref{sec:intro}, \eqref{eqn:cv-def} that $c_t(v) = c_{\BS(t)}(v)$, for $v \in [n]$.    For $t\geq 0$ and $i\in \Omega_K$, define 
\be
X_i(t)=|\set{v \in [n] : c_t(v)=i}| \mbox{ and } \barx_i(t)=X_i(t)/n. \label{eqn:def-x}
\ee
Following \cite{spencer2007birth},
the first step in analyzing bounded-size rules   is understanding the evolution of $\barx_i(\cdot)$ as functions of time as $n\to\infty$. 
Although \cite{spencer2007birth} proves the convergence of $\barx_i(t)$ as $n \to \infty$, we give here a self contained proof of this convergence with precise rates of convergence that will be needed in the proof of 
Theorem \ref{thm:subcrit-reg}.  
 
Note that the BSR process changes values at the occurrence of points in the Poisson processes $\PP_{\vec{v}}$, $\vec{v} \in [n]^4$.  We call each such occurrence as a `round'
and call a round {\bf redundant} if the added edge in that round joins two vertices in the same component. Note that such rounds do not have any effect on component sizes or on the vector $\bar{\bfx}(t) = (\bar x_1(t), \bar x_2(t), \ldots, \bar x_K(t), \bar x_\ompar(t))$. We will in fact observe that such rounds are quite rare. We now describe the effect of non-redundant rounds on $\bar{\bfx}(\cdot)$.  For $\vec{j} \in \Omega_K^4$ and $i \in \Omega_K$, write $\Delta(\vec{j};i)$ for the change $\Delta X_i(t) := X_i(t)-X_i(t-)$ 
at an occurrence time $t$
if the chosen quadruple $\vec{v} \in [n]^4$  satisfies $ c_{t-}(\vec{v}) = \vec{j} $  and the round is not redundant. It is easy to check (see Section 2.1, \cite{spencer2007birth}) that, when $\vec{j}=(j_1,j_2,j_3,j_4) \in F$,
\begin{align*}
	\Delta(\vec{j}; i) &= i\cdot ({\bf 1}_{\{j_1+j_2=i\}} -{\bf 1}_{\{j_1=i\}} - {\bf 1}_{\{j_2=i\}}) , \mbox{ for }  1 \le i \le K,\\
	\Delta(\vec{j}; \ompar) &= \ind_{\{ j_1 + j_2 =\ompar\}} ( j_1 \ind_{\{j_1 \le K\}}+ j_2 \ind_{\{j_2 \le K\}} ),
\end{align*}
with the convention $j_1 + j_2 = \ompar$ when the  sum  of $j_1, j_2$ is greater than $K$, and $j_1 + \ompar = \ompar + j_1 =\ompar$ for all $j_1 \in \Omega_K$. For $\vec{j}=(j_1,j_2,j_3,j_4) \in F^c$  one uses the second edge $\set{v_3, v_4}$ and the expressions for $\Delta(\vec{j};i)$ are  identical to the above, with $(j_3, j_4)$ replacing $(j_1, j_2)$. Note that the corresponding change in the density $\barx_i(t) = X_i(t)/n$ is given by $\Delta(\vec{j};i)/n$. For $\vec{j}\in \Omega_K^4$ and $t> 0$, write  
\[\cQ(t;\vec{j}):= \set{\vec{v}\in [n]^4: c_t(\vec{v}) = \vec{j} }.\]
Since each quadruple $\vec{v}\in [n]^4$ is selected according to the Poisson process $\PP_{\vec{v}}$ with rate $1/2n^3$, the above description of the jumps of $X_i(\cdot)$ leads to a  semi-martingale decomposition of $\barx_i$ of the form \eqref{eqn:semi-mart-def}  
with 
\begin{equation}
\label{eqn:fi-eqn}
	\bfd(\bar x_i)(t) = \sum_{\vec{j}\in F} \sum_{\vec{v}\in \cQ(t; \vec{j})} \frac{\Delta(\vec{j};i)}{2n^4}\ind\set{\CC_{v_1}(t)\neq \CC_{v_2}(t)} + \sum_{\vec{j}\in F^c} \sum_{\vec{v}\in \cQ(t; \vec{j})} \frac{\Delta(\vec{j};i)}{2n^4}\ind\set{\CC_{v_3}(t)\neq \CC_{v_4}(t)},
\end{equation}
where  $\CC_v(t) := \CC_v(\BS(t))$ denotes the component containing $v$ in  $\BS(t)$.  

Define for $i \in \Omega_K$, the functions $F_{i}:[0,1]^{K+1}\to \Rbold$ mapping the vector  $\bfx=(x_1, x_2,...,x_K, x_\ompar) \in \RRR^{K+1}$ to  
\be
F_i^x(\bfx) = \frac{1}{2}\sum_{\vec{j} \in F } \Delta(\vec{j}; i) x_{j_1}x_{j_2}x_{j_3}x_{j_4} + \frac{1}{2}\sum_{\vec{j} \in F^c } \Delta(\vec{j}; i) x_{j_1}x_{j_2}x_{j_3}x_{j_4}. \label{eqn:def-fxi}
\ee 

Note that  $|\Delta(\vec{j};i)|\leq 2K$ for all $\vec{j}\in \Omega_K^4$. Also,
$$
\max \left \{ 
\sum_{\vec{j}\in F} \sum_{\vec{v}\in \cQ(t; \vec{j})} \ind\set{\CC_{v_1}(t)= \CC_{v_2}(t)} , \sum_{\vec{j}\in F^c} \sum_{\vec{v}\in \cQ(t; \vec{j})} \ind\set{\CC_{v_3}(t)= \CC_{v_4}(t)} \right \} \le n^3K.$$
Thus we have 
\be
|\bfd(\bar x_i)(t)- F_i^x(\bar \bfx(t))| \le \frac{2K}{2n^4} \cdot 2 Kn^3 =\frac{2K^2}{n}. \label{eqn:error-xi-trend}
\ee
Note that $\barx_1(0) = 1$ while $ \barx_i(0) = 0$  for other $i\in \Omega_K$.
Guided by equations \eqref{eqn:fi-eqn} -- \eqref{eqn:error-xi-trend}, \cite{spencer2007birth} considered  the system of differential equations for $ \bfx(t) := (x_j(t):j \in \Omega_K )$
\begin{equation}
x_i^\prime(t)= F_i^x(\bfx(t)), \;\; i \in \Omega_K, \;\; t\ge 0,\;	  \bfx(0)=(1,0,..., 0),	\label{eqn:sys-dif-eqns}
\end{equation}
and showed the following result. 
\begin{Theorem}[Theorem 2.1, \cite{spencer2007birth}]
	\label{theo:exist-diff-eqn}
	Equation \eqref{eqn:sys-dif-eqns} has a unique solution. For all $i \in \Omega_K$ and $t> 0$,  $x_i(t)>0$. Furthermore $ \sum_{i\in \Omega_K} x_i(t) = 1$ and $\lim_{t\to \infty} x_\ompar(t) =1$. 
\end{Theorem}
The paper \cite{spencer2007birth}  also showed that  the functions $\barx_i(t) \convp x_i(t)$ for each fixed $t\geq 0$. We will need precise rates of convergence for our proofs for which we establish the following result.
\begin{Lemma}	
	\label{lemma:approx-xi}
Fix $\delta \in (0,1/2)$ and $T> 0$. There exist   $C_1, C_2 \in (0, \infty)$ such that for all $n$, 
	\[\prob\left( \sup_{i \in \Omega_K} \sup_{s \in [0,T]} |\barx_i(t)-x_i(t) | > n^{-\delta}\right) <C_1 \exp\left( -C_2 n^{1-2\delta}\right).\]
\end{Lemma}
\begin{proof}
%
%

 
Note that $F_i^x(\cdot)$ is a Lipchitz function, indeed for $\bfx, \tilde \bfx \in [0,1]^{K+1}$,
$$ |F_i^x(\bfx) - F_i^x(\tilde \bfx)| \le 4K(K+1)^4  \sum_{i \in \Omega_K}|x_i-\tilde x_i| \le 4K(K+1)^5 \sup_{i \in \Omega_K}|x_i-\tilde x_i|.$$
Write  $D(t) := \sup_{i \in \Omega_K}|\bar x_i(t)-x_i(t)|$ and $M_i(t) :=\vM(\bar x_i)(t)$. Using \eqref{eqn:error-xi-trend}, we get for all $i \in \Omega_K$ and $ t \in [0,T]$,
\begin{align*}
	|\bar x_i(t)-x_i(t)| 
	\le& \int_0^t |F_i^x(\bar \bfx(s))- F_i^x(\bfx(s))|ds + T \cdot \frac{2K^2}{n} + |M_i(t)|\\
	\le& 4K(K+1)^5 \int_0^t D(s) ds + T \cdot \frac{2K^2}{n} + |M_i(t)|.
\end{align*}
Taking $\sup_{i\in \Omega_K}$ on both sides and using Gronwall's lemma we have
\bes
\sup_{t \in [0,T]} D(t) \le \left( \sup_{i\in \Omega_K} \sup_{t \in [0,T]}|M_i(t)| + \frac{2TK^2}{n} \right) e^{4K(K+1)^5 T}.
\ees
Thus, for a suitable $d_1 \in (0, \infty)$,
\be
\prob\set{ \sup_{t \in [0,T]} D(t) > n^{-\delta} } \le \sum_{i \in \Omega_K} \prob\set{ \sup_{t \in [0,T]}|M_i(t)| > d_1 n^{-\delta} }. \label{eqn:1964}
\ee
To complete the proof we will use exponential tail bounds for martingales. From Theorem 5 in Section 4.13 of \cite{liptser-mart-book} we have that, for a square integrable martingale $M$ with $M(0) = 0$,   $|\Delta M(t)| \le c$ for all $t$, and  $\langle M, M \rangle(T) \leq Q$, a.s., 
for some $c,Q \in (0, \infty)$, 
$$ \prob\set{ \sup_{0 \le t \le T}|M(t)| > \alpha} \le 2 \exp \set{ -\sup_{\lambda > 0}\left[\alpha \lambda - Q \psi(\lambda)\right] }, \mbox{ for all } \alpha > 0,$$
where $\psi(\lambda) = \frac{e^{\lambda c}-1-\lambda c}{c^2}$. Optimizing over $\lambda$, we get the bound
\be
\prob\set{ \sup_{0 \le t \le T}|M(t)| > \alpha} \le 2 \exp \set{ - \frac{\alpha}{c} \log\left( 1+ \frac{\alpha c}{Q} \right) + \left[ \frac{\alpha}{c} - \frac{Q}{c^2} \log\left( 1+ \frac{\alpha c}{Q} \right)  \right] }. \label{eqn:norris-bound}
\ee
    In our context, note that for any $i\in \Omega_K$,  $| \Delta M_i(t)|=|\Delta \bar x_i(t)| \le 2K/n$. Also, the total rate of jumps is bounded by $n^4 \cdot \frac{1}{2n^3}$. Thus for all $i\in \Omega_K$, the quadratic variation process 
$$ \langle M_i, M_i \rangle(T) \leq   \int_0^T \left(\frac{2K}{n}\right)^2 \times \frac{n^4}{2n^3} dt = \frac{2K^2T}{n}.$$
Taking $\alpha = d_1 n^{-\delta}, Q={2K^2T}/{n}$ and $c = 2K/n$ in \eqref{eqn:norris-bound} completes the proof. 
\end{proof}

\subsection{Evolution of components of size larger than $K$}
\label{sec:evol-bst-st}
Let $\BS^*(t)$ denote the subgraph of $\BS(t)$ consisting of components of size greater than $K$. In this section, we will focus on the dynamics and evolution of $\BS^*(t)$. Note that $\BS^*(0) = \emptyset$, i.e. a graph with no vertices or edges. As time progresses components of size less than $K$ merge and components of size greater than $K$ emerge. Three distinct types of events affect the evolution $\BS^*(t)$:
\begin{enumeraten}
	\item {\bf Immigration:} This occurs when two components of size $\le K$ merge into a single component of size $> K$. We view the resulting component as  a new immigrant into $\BS^*(t)$. Note that the first component to appear in $\BS^*(t)$ is an immigrant. 
	\item {\bf Attachments}: This occurs when a component of size $\le K$ gets linked to a component of size larger than $K$. The former component enters $\BS^*(t)$ via attaching itself to a component of size larger than $K$.
	\item {\bf Edge formation}: Two distinct components of size larger than $ K$  merge into a single component via formation of an edge between these components. In this case, the vertex set of $\BS^*(t)$ remains unchanged.
\end{enumeraten}
We now introduce some functions that describe the rate of occurrence for each of the three types of events. For $i_1,i_2 \in \Omega_K$, define $F_{i_1,i_2}^x : [0,1]^{K+1} \to \RRR$ as
\be
\label{def:fij}
F_{i_1,i_2}^x(\bfx) = \frac{1}{2}\left( \sum_{\vec{j} \in F: \{j_1,j_2\}=\{i_1,i_2\}} x_{j_1}x_{j_2}x_{j_3}x_{j_4} + \sum_{\vec{j} \in F^c: \{j_3,j_4\}=\{i_1,i_2\}} x_{j_1}x_{j_2}x_{j_3}x_{j_4}\right).
\ee
For $i_1, i_2 \leq K$, denote $n \cdot R_{i_1, i_2}(t)$ as the rate at which two components of size $i_1, i_2$ merge. When $i_1 \neq i_2$, this rate
 is precisely 
\begin{align*}
(2n^3)^{-1}[\hspace{-.2in} \sum_{\sss\substack{\vec{j} \in F\\ \{j_1,j_2\}=\{i_1,i_2\}}}\hspace{-.3in} X_{j_1}(t)X_{j_2}(t)X_{j_3}(t)X_{j_4}(t) + \hspace{-.3in} \sum_{\sss\substack{\vec{j} \in F\\ \{j_3,j_4\}=\{i_1,i_2\}}}\hspace{-.3in} X_{j_1}(t)X_{j_2}(t)X_{j_3}(t)X_{j_4}(t)]:= n \cdot F_{i_1,i_2}^x(\bar \bfx(t)).
\end{align*}
Thus $R_{i_1, i_2}(t) = F_{i_1,i_2}^x(\bar \bfx(t))$. 
The case $i_1=i_2 \le K$ is more subtle due to redundant rounds linking vertices in the same component. The 
 rate of redundant rounds can be bounded by $ \frac{1}{2n^3} \cdot K n^3 \cdot 2=K$, from which it follows that 
$$|R_{i,i}(t) -  F_{i,i}^x(\bar \bfx(t))| \le \frac{K}{n}.$$
The case $i_1 = i_2 =\ompar$ corresponds to  creation of edges in $\BS^*(t)$
and $n \cdot F_{\ompar,\ompar}^x(\bar \bfx(t)) $ is the rate of creation of such edges.

We now give expressions for the  rates for the three types of events that govern the evolution of  $\BS^*(t)$. The convention followed for the rest of this section is that for $i_1, i_2\in\Omega_K$,  $ i_1 + i_2 = \ompar$ when the sum of is greater than $K$, and $\ompar + i_1 = i_1 +\ompar = \ompar$ for all $i_1 \in \Omega_K$.

\noindent{\bf I. Immigrating vertices: } For $1 \le i \le K$, write $\fI_i(t) := n \cdot a^*_i(t)$ for the rate at which components of size $K+i$ immigrate into $\BS^*(t)$ at time $t$. Using the above expressions for the rate of merger of components of various sizes we have 
\begin{equation}
\label{eqn:ai-diff}
	\left|a^*_i(t) - \sum_{ 1 \le i_i, i_2 \le K: i_1 +i_2=K+i} F_{i_1,i_2}^x(\bar \bfx(t))\right| \le \frac{K}{n}.
\end{equation}  
As before the error is due to redundant rounds which can only occur for $i_1 = i_2 = (K+i)/2$ (and when  $(K+i)/2$ is an integer). Now define  functions $F^a_i : [0,1]^{K+1} \to \RRR_+$, and $a_i(\cdot): [0,\infty) \to [0, \infty)$ by
\be
F^a_i(\bfx) = \sum_{ \sss \substack{1 \le i_i, i_2 \le K, \\ i_1 +i_2=K+i}} F_{i_1,i_2}^x(\bfx), \qquad a_i(t) = F^a_i(\bfx(t)), \label{eqn:def-fa-a}
\ee
where $\bfx(t)$ is as in \eqref{eqn:sys-dif-eqns}.
Then \eqref{eqn:ai-diff} says that 
\be
\sup_{t \in [0,\infty) }| a^*_i(t)-F^a_i(\bar \bfx(t))| \le K/n.	\label{eqn:error-a}
\ee
Note that for any $\delta <1$, the error term in \eqref{eqn:error-a} is $o(n^{-\delta})$. Using this observation along with the Lipschitz property of $F_{i_1, i_2}$, we have
from Lemma \ref{lemma:approx-xi}  that for any fixed $T> 0$ and $\delta < 1/2$,
\begin{equation}
\label{eqn:ai-aist-diff}
	\prob(\sup_{1\leq i\leq K}\sup_{s\in [0,T]}|a_i^*(t) - a_i(t)| > n^{-\delta}) \leq C_1 \exp(-C_2 n^{1-2\delta}).
\end{equation}
The constants $C_1, C_2$ here may be different from those in Lemma \ref{lemma:approx-xi}, however for notational ease we use the same symbols.

\noindent{\bf II. Attachments:} Fix $1\le i\le K$ and  a vertex $v$ contained in a component in $\BS^*(t)$.  Let, for $i \le K$, $c_i^*(t)$ denote the rate at which a component of size $i$ attaches itself to the component of $v$ through an edge connecting the former component to $v$. 
This rate can be calculated as follows. First note that the total rate of merger between a component of size $i$ with a component in $\BS^*(t)$ is $n \cdot F_{i,\ompar}^x(\bar \bfx(t))$. Since there are $X_\ompar(t)$  vertices in
$\BS^*(t)$ each of which has the same probability of being the vertex through  which this attachment event happens, the rate at which a component of size $i$ attaches to $v$ is given by $n F_{i,\ompar}^x(\bar \bfx(t)) / X_\ompar(t) = F_{i,\ompar}^x(\bar \bfx(t))/{\bar x}_\ompar(t) $. 
Since $x_\ompar$ is a factor of $F_{i,\ompar}^x(\bfx)$, $c^*_i(t)$ is  a polynomial in ${\bar \bfx}(t)$. Define the functions $F^c_i : [0,1]^{K+1} \to \RRR$ and $c_i(\cdot): \Rbold_+ \to \Rbold_+$ as  
\be
F^c_i(\bfx) = F_{i,\ompar}^x(\bfx)/x_\ompar, \qquad  c_i(t) = F^c_i(\bfx(t)). \label{eqn:def-fc-c}
\ee
Then $c_i^*(t) = F_i^c({\bar \bfx}(t))$. Once again using Lemma \ref{lemma:approx-xi} we get for any $\delta <1/2$ and  $T> 0$,
\begin{equation}
\label{eqn:ci-cist-diff}
	\prob(\sup_{1\leq i\leq K}\sup_{s\in [0,T]}|c_i^*(t) - c_i(t)| > n^{-\delta}) \leq C_1 \exp(-C_2 n^{1-2\delta}). 
\end{equation}


\noindent{\bf III. Edge formation:}
Note that the rate of creation of an edge between vertices in $\BS^*(t)$ is $n F_{\ompar, \ompar}^x(\bar \bfx(t))$. Since such an edge is equally likely  to be between any of the $X_\ompar^2(t)$ pairs of vertices in $\BS^*(t)$, we have that 
 the rate of creation of an edge between specified vertices  $\set{v_1 , v_2}$ with $v_1, v_2 \in \BS^*(t)$ is $b^*(t)/n$ where $b^*(t) = F_{\ompar, \ompar}^x({\bar\bfx}(t))/x_\ompar^2(t)$.   
Define $F^b : [0,1]^{K+1} \to \RRR$ and $b(\cdot) : \Rbold_+ \to \Rbold_+$ as
\be
F^b(\bfx) = F_{\ompar, \ompar}^x(\bfx)/ x^2_\ompar \mbox{ and } b(t)=F^b(\bfx(t)). \label{eqn:def-fb-b}
\ee
Once more it is clear that $F^b(\bfx)$ is a polynomial  and furthermore $b^*(t) = F^b({\bar \bfx}(t))$, so  by Lemma \ref{lemma:approx-xi}, for any $\delta <1/2$ and  $T> 0$,
\begin{equation}
\label{eqn:bi-bist-diff}
	\prob(\sup_{s\in [0,T]}|b^*(t) - b(t)| > n^{-\delta}) \leq C_1 \exp(-C_2 n^{1-2\delta}). 
\end{equation}
Write
$\va(t):=\set{a_i(t)}_{1\leq i\leq K}$ and $ \vc(t):=\set{c_i(t)}_{1\leq i\leq K} $.  We refer to $(\va, b, \vc)$  as rate functions.  In the proposition below we
 collect some properties of these rate functions. These properties  are easy consequences of  Theorem \ref{theo:exist-diff-eqn}. 
\begin{Proposition}
		\label{prop:error-abc}
	\begin{enumeratea}
		\item For all $1\leq i\leq K$ and $t> 0$, $b(t), a_i(t), c_i(t)> 0$.
		\item We have
		 $$ \|\va\|_\infty := \sup_{t \ge 0} \sum_{i=1}^K a_i(t) \le 1/2, \;\; \|\vc\|_\infty := \sup_{t \ge 0}\sum_{i=1}^K c_i(t) \le 1/2,\;\; \|b\|_\infty := \sup_{t \ge 0} b(t) \le 1/2 .$$
		\item  $\lim_{t \to \infty} b(t) = 1/2$.
	\end{enumeratea}
\end{Proposition}
{\bf Proof:} Part(a) follows from Theorem \ref{theo:exist-diff-eqn} and the definition of the functions. For (b) observe that 
$$ \sum_{i=1}^K a_i(t) = \sum_{i=1}^K F^a_i(\bfx(t)) \le \frac{1}{2}\sum_{ \vec{j} \in \Omega_K} x_{j_1}(t)x_{j_2}(t)x_{j_3}(t)x_{j_4}(t)= \frac{1}{2}\left[\sum_{i\in \Omega_K}x_i(t)\right]^4 =\frac{1}{2}. $$
Statements on $\|c\|_\infty,  \|b\|_\infty$ follow similarly. 

For (c), note that $F_{\ompar,\ompar}^x(\bfx) \ge x_\ompar^4/2$ since when all the four vertices selected are  from components of size greater than $K$, two components of  size greater than $K$ will surely be linked. From Theorem \ref{theo:exist-diff-eqn} 
 $\lim_{t \to \infty} x_\ompar(t)=1$ and thus  $ \limsup_{t\to \infty} b(t) \ge x^2_\ompar (t)/2$. The result now follows on combining this with
(b). \qed


\subsection{Connection to inhomogeneous random graphs}
\label{sec:other-random-graph-processes}
In this section, we describe the inhomogeneous random graph (IRG) models that have been studied  extensively in \cite{bollobas-riordan-janson}, and then approximate $\BS^*(t)$   by a special class of such models. 
We will in fact use a variation of the models in \cite{bollobas-riordan-janson} which uses a suitable weight function to measure the volume of a component.  
We begin by defining the basic ingredients in this model. Let $\XX$ be a Polish space, equipped with the Borel $\sigma$-field $\cB(\cX)$. We shall sometimes refer to this as the {\bf type space}. Let $\mu$ be a non-atomic finite measure on $\XX$ which we shall call the {\bf type measure} on $\cX$. A {\bf kernel} will be a symmetric non-negative product measurable function $\kappa: \XX \times \XX \to \RRR $ and a {\bf weight function} $\phi: \XX \to \RRR$ will be a non-negative measurable function on $\XX$. We call such a quadruple $\{\XX, \mu, \kappa, \phi\}$ a {\bf basic structure}. 

{\bf The inhomogeneous random graph with weight function (IRG):} 
Associated with a basic structure $\{\XX, \mu, \kappa, \phi\}$, 
the IRG model $\RG^{\sss (n)}(\XX, \mu, \kappa, \phi)$ is a random graph described as follows: 
\begin{enumeratea}
	\item {\bf Vertices:} the vertex set $\VV$ of this random graph is a Poisson point process on the space $\XX$ with intensity measure $n \mu$.
	\item {\bf Edges:}  an edge is added between vertices $\bfx, \bfy \in \VV$ with probability $ 1 \wedge \frac{\kappa(\bfx,\bfy)}{n}$, independent across different pairs. This defines the random graph. 
	\item {\bf Volume:} The volume of a component $\CC$ of $\RG^{\sss(n)}(\cX,\mu,\kappa,\phi)$ is defined as 
	\begin{equation}
	\label{eqn:vol-defn}
	\vol_{\phi}(\CC):=\sum_{x \in \CC} \phi(x). 	
	\end{equation}
 \end{enumeratea}
For the rest of this section we  take
\begin{equation}
	\label{eqn:def-xx}
	\XX:= [0,\infty) \times \WW \text{ where } \WW:=\DD([0, \infty): \NNN_0).
\end{equation}
We first describe how, for each $t >0$, $\BS^*(t)$ can be identified with a random graph with vertex set in $\XX$. 
%
Recall the three  types of events governing the evolution of $\BS^*(t)$, described in Section \ref{sec:evol-bst-st}. 
Each component in $\BS^*(t)$ contains at least one group of $K+i$ vertices, $i = 1, \cdots , K$ which appeared at instant   $s \le t$ in $\BS^*(\cdot)$, as an immigrant.
We denote the collection of all such groups 
as $\Imm(t)$.  For $\clc \in \Imm(t)$, we denote
by $\tau_\clc\in (0, t]$ the instant this immigrant appears.  Also, to each $\clc \in \Imm(t)$, we associate a path in $\cD([0,\infty): \NNN_0)$, denoted as
$w_{\clc}$, such that $w_{\clc}(s) = 0$ for all $s < \tau_\clc$; $w_{\clc}(s) = w_{\clc}(t)$ for all $s \in [t,\infty)$; and for $s \in [\tau_\clc,t]$,
$w_{\clc}(s) = | \cla_{\clc}(s)|$, where $\cla_{\clc}(s)$ denotes the component that is formed by $\clc$ and all the attachment components that link
to $\clc$ over the time interval $[\tau_\clc,s]$.
Then $\{(\tau_\clc, w_{\clc}): \clc \in \Imm(t)\}$ is a point process on $\XX$ and forms the vertex set of a random graph which we denote by $\mirror(t)$.
We form edges between any two vertices $(\tau_\clc, w_{\clc})$, $(\tau'_\clc, w'_{\clc})$ in $\mirror(t)$ if the components $\cla_{\clc}(t)$
and $\cla_{\clc'}(t)$ are directly linked by some edge in $\BS^*(t)$.\\

Define, for $t >0$,  the weight function $\phi_t: \XX \to [0, \infty)$ as
\begin{equation}
	\label{eqn:def-phi}
	\phi_t(\bfx) = \phi_t(s,w) = w(t), \qquad \bfx=(s,w) \in \cX.
\end{equation}
Note that by construction there is a one to one correspondence between the components in $\BS^*(t)$ and the components in
$\mirror(t)$.  
For a component $\clc_0$ in $\BS^*(t)$, denote by $I_{\clc_0}$ the corresponding component in $\mirror(t)$.  Then note that
\begin{equation}\label{eq:eq1848}
 |\clc_0| = \vol_{\phi_t}(I_{\clc_0}) .
\end{equation}


We will now describe inhomogeneous random graph models that approximate $\mirror(t)$ (and hence $\BS^*(t)$).
Given a set of nonnegative continuous bounded functions $\mvalpha = \set{\alpha_i}_{1\leq K}$, $\beta$ and $\mvgamma = \set{\gamma_i}_{1\leq i\leq K}$ 
on $[0,\infty)$ we construct, for each $t >0$, type measures
$\mu_t(\mvalpha, \beta, \mvgamma)$ and kernels $\kappa_t(\mvalpha, \beta, \mvgamma)$ on $\XX$ as follows.
For $i = 1, \cdots K$ and $s \in [0,\infty)$, denote by $\tilde \nu_{s,i}$ the probability law on $\cD([s,\infty): \NNN_0)$ of the Markov process
$\{\tilde w(r)\}_{r \in [s,\infty)}$ with infinitesimal generator 
\begin{equation}
	 (\calA(u) f) (k) =  \sum_{j=1}^K k \gamma_j(u)  (f(k+j)-f(k)),\; f \in \BM(\NNN_0)	\label{eqn:attachment-operator}
\end{equation}
and initial condition $\tilde w(s) = K+i$. In words, this is a pure jump Markov process which starts at time $s$ at state $K+i$ and then at any time instant $u> s$, has jumps of size $j$ at rate $\gamma_j(u)$. 
Denote by $ \nu_{s,i}$ the probability law on $\cD([0,\infty): \NNN_0)$   of the stochastic process
$\{ w(r)\}_{r \in [0,\infty)}$, defined as
\begin{equation}
\label{eqn:nu-si-defn}
	w(r) = \tilde{w}(r) \quad\mbox{ for $r\geq s$ },\qquad w(r) = 0 \mbox{ otherwise. }
\end{equation}
Now define the finite measure $\mu_t(\mvalpha, \beta, \mvgamma) \equiv \mu_t$ as 
\begin{equation}
	\label{eqn:def-mu-gen}
	\int_\XX f(\bfx) d\mu_t(\bfx) = \sum_{i=1}^K \int_0^t \alpha_i(u) \left( \int_\WW f(u, w)\nu_{u,i}(dw) \right) du, \; f \in \BM(\XX).
\end{equation}
Next, define the kernel $\kappa_t(\mvalpha, \beta, \mvgamma) \equiv \kappa_t$ on $\XX \times \XX$ as
\begin{equation}
	\label{eqn:def-kappa-gen}
	\kappa_t(\bfx, \bfy)=\kappa_{t}((s,w),(r,\tilde w))=\int_0^t w(u) \tilde w(u) \beta(u) du, \; \bfx=(s,w), \bfy=(r,\tilde w) \in \XX .
\end{equation}
With the above choice of $\mu_t$, $\kappa_t$ and with weight function $\phi_t$ as in \eqref{eqn:def-phi} we now construct the random graph
$\RG^{\sss (n)}(\XX, \mu_t, \kappa_t, \phi_t)$  which we denote by $\RG^{\sss(n)}(\mvalpha,\beta,\mvgamma)(t)$.  We will refer to the set of
functions $(\mvalpha,\beta,\mvgamma)$ as above, as {\bf rate functions}.  These rate functions will typically arise as small perturbations of the  functions
$(\va , b, \vc)$, thus in view of Proposition \ref{prop:error-abc}(b) it will suffice to consider $(\mvalpha,\beta,\mvgamma)$ such that $\max\{||\mvalpha||_{\infty},||\beta||_{\infty},
||\mvgamma||_{\infty}\} < 1.$  Throughout this work we will assume that all rate functions  (and their perturbations) satisfy this bound.\\

The following key result says that for large $n$, $\mirror(t)$ is suitably close to $\RG(\va , b, \vc)(t)$, where $(\va , b, \vc)$ are the rate functions
introduced below \eqref{eqn:bi-bist-diff}. In order to state the result precisely, we extend the notion of a ``subgraph'' to the  setting with  type space $\XX$ and  weight function $\phi$. For $i=1,2$, consider graphs $\bfG_i$, with finite vertex set $\VV_i \subset \XX$ and edge set $\EE_i$. We say $\bfG_1$ is a subgraph of $\bfG_2$, and write $\bfG_1 \subset \bfG_2$, if there exists a  one to one mapping $\Psi: \VV_1 \to \VV_2$ such that \\
(i) $\phi(\bfx) \le \phi (\Psi(\bfx)), \mbox{ for all }\bfx \in \VV_1$.\\
(ii) $\{\bfx_1, \bfx_2\} \in \EE_1$ implies $\{\Psi(\bfx_1), \Psi(\bfx_2)\} \in \EE_2$.\\
\begin{Lemma}
	\label{lemma:bsr-to-irg}
	Fix $\delta \in (0, 1/2)$ and let $\eps_n = n^{-\delta}$, $n \ge 1$.  Define, for $t >0$, the set of functions
$\va^-(t):= \set{(a_i(t)-\eps_n)\vee 0 }_{1\leq i\leq K}$, $\va^+(t):= \set{ a_i(t)+\eps_n}_{1\leq i\leq K}$ and similarly $\vc^-(t), \vc^+(t)$ and $b^-(t), b^+(t)$. 
Define the inhomogeneous random graphs (IRG) with the above rate functions by
$$\RG^-(t) := \RG(\va^-,b^-,\vc^-)(t) , \;\; \RG^+(t):=\RG(\va^+,b^+,\vc^+)(t). $$

Then for every $T > 0$ there exist $C_3, C_4 \in (0, \infty)$, such that for all $t \in [0,T]$, there is a coupling of $\RG^-(t) $, 	$ \RG^+(t)$ and 
$\mirror(t)$  such that,  
	$$ \prob \set{ \RG^-(t) \subset \mirror (t) \subset \RG^+(t) } > 1 - C_3 \exp \set{ -C_4 n^{1-2\delta} }.$$
\end{Lemma}
\textbf{Proof:} The coupling between the three graphs is done in a manner such that $\Gamma(t)$ is obtained by a suitable thinning of vertices and edges in $\RG^+(t)$
and $\RG^-(t)$ is obtained by a thinning of $\Gamma(t)$.  We will only provide details of the first thinning step.  We first construct the vertex sets 
$\clv^+$ and $\clv^*$ in $\RG^+(t)$ and $\Gamma(t)$ respectively.  

Let $\clv^+$ be a Poisson point process on $\XX$ with intensity $n \mu_t^+$, where $\mu_t^+ := \mu_t(\va^+,b^+,\vc^+)$.
For a fixed realization of $\clv^+$, denote by
$(x_1^+, \cdots, x_N^+)$, the points in $\clv^+$, with
$x_i^+ = (s_i^+, w_i^+)$ and $0 < s_1^+ < s_2^+ \cdots s_N^+ < t$.  Write $\bfw^+ = (w_1^+, \cdots w_N^+)$.
We now construct vertices in the corresponding realization of $\Gamma(t)$ (denoted as $\{x_1, \cdots x_{N_0}\}$), along with the realizations of
$\bar x_i(s)$, $i \in \Omega_K$, $0 \le s \le t$, which then defines 
 $(a^*_j(s), b^*(s), c^*_j(s))$, $0 \le s \le t$, $j = 1, \cdots K$, as functions of $\bar \bfx(s) = (\bar x_i(s))_{i \in \Omega_K}$ as in Section \ref{sec:evol-bst-st}.
For that, we will construct functions $w_j: [0,t] \to \NNN_0$, $1 \le j \le N$ and $\bar x_i: [0,t] \to [0,1]$, $i \in \Omega_K$.
We will only describe the construction of $w_j, \bar x_i$ until the first time instant $s \in (0, t]$, when the property
\begin{equation}
	\label{eqn:couple-failure}
	 a^*_j(s) \le a^+_j(s), \; b^*(s) \le b^+(s), \;c^*_j(s) \le c^+_j(s) \mbox{ for all } 1 \le k \le K
\end{equation}
is violated. Denote $\sigma$ for the first time that \eqref{eqn:couple-failure} is violated with $\sigma$ taken to be $t$ if the property holds for all $s \in [0,t]$.  Subsequent to that time instant the construction can be done in any  fashion that yields the correct probability law for $\Gamma(t)$.
For simplicity, we assume henceforth that $\sigma = t$.
After obtaining the functions $w_j, \bar x_i$, we set 
$x_i^* = (\tau_i^*, w_i^*)$, where $\tau_i^*$ is the first jump instant of $w_i$ (taken to be $+\infty$ if there  are no jumps) and
$w_i^* \in \DD([0, \infty): \NNN_0)$ is defined as
$w_i^*(s) = w_i(s)1_{[0, t]}(s) + w_i(t)1_{(t, \infty)}(s)$. The vertex set $\clv^*$ is then defined as
$$
\clv^* = \{x_1, \cdots x_{N_0}\} = \{x_i^*: \tau_i^* < t, i = 1, \cdots N\}.$$

We now give the construction of $\bfw(s) = (w_1(s), \cdots w_N(s))$ and $\bar \bfx(s)$ for $s \le t$.  
Denote by $\{t_i\}_{i=1}^M$, $0=t_0 < t_1 < t_2 < ...t_M < t$,
the collection of all time instants of jumps of $\{w_i^+\}_{i=1}^N$ before time $t$.  Denote by  $i_k$ the  coordinate of $\bfw^+$ that has a jump at  time $t_k$,  and denote the
corresponding jump size by $j_k$.
We set
$\bfw(0) = 0$, $\bar x_i(0) = 0$ for $i \neq 1$ and $\bar x_1(0) = 1$.
The construction proceeds recursively over the time intervals $(t_{k-1}, t_{k}]$, $k =  1, \cdots M+1$, where $t_{M+1} =t$.
Suppose that $(\bfw(s), \bar \bfx(s))$ have been defined for $s \in [0, t_{k-1}]$, for some $k\ge 1$.
We now define these functions over the interval $(t_{k-1}, t_k]$.\\

\noindent \textbf{Step 1:} $s \in (t_{k-1},t_k)$.  Set $\bfw(s) = \bfw(t_{k-1})$.  The values of $\bar \bfx(s)$ over the interval will be given
as a realization of a jump process, for which jumps at time instant $s$ occur at rates
$n \cdot R_{i_1, i_2}(s)$, $i_1, i_2 \in \{1, \cdots K\}$, $i_1+i_2 \le K$, where the function $R_{i_1,i_2}(s)$, given as a function of $\bar \bfx(s)$ is defined as in Section
\ref{sec:evol-bst-st}.  A jump at time instant $s$, corresponding to the pair $(i_1, i_2)$ as above, changes the values of $\bar \bfx$ as:
\begin{align*}
	&\bar x_{i_1}(s)= \bar x_{i_1}(s-)-\frac{i_1}{n},\; \bar x_{i_2}(s)= \bar x_{i_2}(s-)-\frac{i_2}{n}, \; \bar x_{i_1+i_2}(s)= \bar x_{i_1+i_2}(s-)+\frac{i_1+i_2}{n}, \;\; \mbox{if}\; i_1 \neq i_2,\\
	&\bar x_{i_1}(s)= \bar x_{i_1}(s-)-\frac{2i_1}{n},\; \bar x_{i_1+i_2}(s)= \bar x_{2i_1}(s-)+\frac{2i_1}{n}, \;\; \mbox{if}\; i_1 = i_2.
\end{align*}
Remaining $\bar x_i$ stay unchanged.  The values of $a^*_i(s)$, $b^*(s)$, $c^*_i(s)$ are determined accordingly.\\

\noindent \textbf{Step 2:} $s= t_k$.  Recall that $w_{i_k}^+(t_k)-w_{i_k}^+(t_k-)=j_k.$  We define
$w_i(t_k) = w_i(t_k-)$ for all $i \neq i_k$.  The values of $w_{i_k}(t_k)$
and $\bar \bfx(t_k)$ are determined as follows.\\

\noindent {\bf Case 1:}  $w_{i_k}^+(t_k-)=0$.
In this case 
$K+1 \le j_k \le 2K$ and $t_k$ is the first jump of $w_{i_k}^+$.
Define for $ 1\le l\le K$, $Q_k(l) := \sum_{i=1}^{l}R_{i,j_k-i}(t_k-)$, where by definition
$R_{i,i'} = 0$ if $i' > K$.  Note that $Q_k(K)= a^*_{j_k}(t_k-)$.  We set $Q_k(0)=0$.
The values of $w_{i_k}(t_k)$
and $\bar \bfx(t_k)$ are now determined according to the realization of an independent Uniform $[0,1]$ random variable $u_k$ as follows.
\begin{itemize}
	\item If $u_k > Q_k(K)/a_{j_k}(t_k-)$, define $ (w_{i_k}(t_k), \bar \bfx(t_k)) = (w_{i_k}(t_k-), \bar \bfx(t_k-))$.
	\item Otherwise, suppose $1 \le l_k \le K$ is such that $ Q_k(l_k-1) < u_k \le Q_k(l_k)$. Then define 
	$w_{i_k}(t_k)=j_k$,  $\bar x_{\ompar}(t_k) = \bar x_{\ompar}(t_k-) + \frac{j_k}{n}$ and
	\begin{align*} 
		\bar x_{l_k}(t_k) &= \bar x_{l_k}(t_k-) - \frac{l_k}{n},\;\;\bar x_{j_k-l_k}(t_k) = \bar x_{j_k-l_k}(t_k-) - \frac{j_k-l_k}{n}, 
		\;\; \mbox{ if } l_k \neq j_k-l_k,\\
		\bar x_{l_k}(t_k) &= \bar x_{l_k}(t_k-) - \frac{2l_k}{n},\;\;
		\;\; \mbox{ if } l_k = j_k-l_k.
	\end{align*}
	The value of all other $x_i$ processes at $t_k$ stay the same as their values at $t_k-$. 
\end{itemize}

\noindent {\bf Case 2:}   $w_{i_k}^+(t_k-) \ne 0$.
In this case
$1\le j_k \le K$.
Once again, the values of $w_{i_k}(t_k)$
and $\bar \bfx(t_k)$ are  determined according to the realization of an independent Uniform $[0,1]$ random variable $u_k$ as follows.
\begin{itemize}
	\item If $u_k > \frac{w_{i_k}(t_k-)c_{j_k}^*(t_k-)}{w_{i_k}^+(t_k-)c_{j_k}^+(t_k-)}$, define $ (w_{i_k}(t_k), \bar \bfx(t_k)) = (w_{i_k}(t_k-), \bar \bfx(t_k-))$.
	\item Otherwise, 
	$$w_{i_k}(t_k)=w_{i_k}(t_k-) + j_k, \;\;\bar x_{j_k}(t_k) = \bar x_{j_k}(t_k-) - \frac{j_k}{n}, \;\;\bar x_{\ompar}(t_k) = \bar x_{\ompar}(t_k-) + \frac{j_k}{n},$$
	and the value of all other $x_i$ processes stay the same as their value at $t_k-$.
\end{itemize}
This completes the construction of $(\bfw(s), \bar \bfx(s))$ for $s \in (t_{k-1}, t_k]$ and thus by this recursive procedure and our earlier discussion
we obtain the vertex set
$$\clv^* = \{x_1, \cdots x_{N_0}\} = \{x_i^*: \tau_i^* < t, i = 1, \cdots N\}, $$ 
which will be used to construct $\Gamma(t)$.

Having constructed vertex sets $\clv^+$ and $\clv^*$, we now construct edges.  For this we take realizations of independent Uniform $[0,1]$ random variables
$\{u_{i,j}\}_{1 \le  i < j < \infty}$ and construct edge between vertices $x_i^+$ and $x_j^+$ in $\clv^+$ if
$$ u_{i,j} \le \frac{1}{n} \int_0^t b^+(s)w_i^+(s)w_j^+(s)ds. $$
This completes the construction of $\RG^+(t)$.  Finally construct an edge between $x_i^*$ and $x_j^*$ if both vertices are in $\clv^*$ and
$$ u_{i,j} \le 1 - \exp \set{-\frac{1}{n}\int_0^t b^*(s)w_i(s)w_j(s)ds}.$$
This completes the construction of $\Gamma(t)$.  
By construction $ \mirror(t) \subset \RG(\va^+,b^+,\vc^+)(t)$ on the set $\sigma =t$.
Also, from \eqref{eqn:ai-aist-diff}, \eqref{eqn:bi-bist-diff} and \eqref{eqn:ci-cist-diff} it follows that
$\prob(\sigma < t) \le C_3 \exp \set{ -C_4 n^{1-2\delta} }$ for suitable constant $C_3, C_4$.  The result follows.  \qed

The following is an immediate corollary of Lemma \ref{lemma:bsr-to-irg}.  
\begin{Corollary}
	\label{cor:cor2127}
	Fix $T > 0$.  Then with  $C_3, C_4 \in (0, \infty)$ and , for $t \in [0,T]$,   a coupling of $\RG^-(t) $,	$ \RG^+(t)$ and 
	$\mirror(t)$  as in Lemma \ref{lemma:bsr-to-irg}:
	\begin{align}
	\label{eqn:bsr-irg-bound}
		&\prob\set{\vol_{\phi_t}(\cI_1^{\sss \RG^-}(t))\leq \vol_{\phi_t}(\cI_1^{\sss \mirror}(t)) \leq \vol_{\phi_t}(\cI_1^{\sss \RG^+}(t))  }\nonumber\\
		 &\;\;\;\;\geq 1-C_3 \exp(-C_4 n^{1-2\delta}),
	\end{align}
	where $\cI_1^{\sss \mirror}(t)$ denotes the  component in $\Gamma(t)$ with the largest volume with respect to the weight function $\phi_t$, and $\cI_1^{\sss \RG^-}(t)$, $\cI_1^{\sss \RG^+}(t)$ are defined similarly.
\end{Corollary}

\section{Proof of the main results}
\label{sec:bsr-largest-component}

In this section, we will complete the proof of Theorems \ref{thm:subcrit-reg} and \ref{theo:new-crit-time}. 
Proof of Theorem \ref{theo:new-crit-time} is given in Section \ref{sec:proof-phase-transition} while proof of Theorem \ref{thm:subcrit-reg}
is given in Section \ref{sec:ia-largest-com}.
Recall that Lemma \ref{lemma:bsr-to-irg} says that $\BS^*$ can be approximated by  $\RG(\va,b,\vc)$. Sections
\ref{sec:proof-norm-abc} and \ref{sec:proof-norm-t} 
analyze properties of integral operators associated with 
$\RG(\mvalpha, \beta, \mvgamma)$ for a general family of rate functions $(\mvalpha, \beta, \mvgamma)$. 
We begin in Section \ref{sec:preliminary-norm}
by presenting some results about an  IRG model $\RG^{\sss(n)}(\cX,\mu,\kappa,\phi)$ on a general type space $\XX$.

\subsection{Preliminary lemmas}
\label{sec:preliminary-norm}
In this section, we collect some results about the general IRG model $\RG^{\sss(n)}(\cX,\mu,\kappa,\phi)$.
Let $\KK$ be the integral operator associated with $(\kappa , \mu)$, as defined in Section \ref{sec:intro}.
Recall that the operator norm of $\KK$, denoted as 	$\|\KK\|$, is defined as
\begin{equation}
\label{eqn:operator-norm-def}
	\|\KK\| = \sup_{ f \in L^2(\cX,\mu), f \neq 0} \frac{ \|\KK f\|_2}{\|f\|_2},
\end{equation}
where for $f \in L^2(\cX,\mu)$, $\|f\|_2 = \left(\int_{\XX}|f(\bfx)|^2 \mu(d\bfx)\right)^{1/2}$.

\begin{Lemma}
\label{lemma:irg-unbounded} 
Fix $(\cX,\mu,\kappa,\phi)$.
Denote the vertex set of  $\RG^{\sss(n)}(\cX,\mu,\kappa,\phi) \equiv \RG^{\sss(n)}$ by $\PP_n$ which is a rate $n\mu$ Poisson point process on $\cX$. 
Let $\KK$ be the integral operator associated with $(\kappa , \mu)$.
Suppose that $\|\KK\| < 1$ and let $\Delta = 1-\|\KK\|$.
Denote by $\cI^{\sss \RG}_1$ the component in $\RG^{\sss(n)}$ with the largest volume (with respect to the weight function $\phi$).
Then the following hold.
\begin{enumeratei}
	\item If $\|\phi\|_\infty < \infty$ and
	$\|\kappa\|_\infty < \infty$, then for all $m \in \NNN$ and $D \in (0, \infty)$
	\begin{equation}
	\label{ins1531}
	\prob \{ \vol_{\phi}(\cI^{\sss \RG}_1) > m \} \le 2 nD\exp \{-C\Delta^2 m\} + \prob(|\PP_n| \ge nD),
	\end{equation}
	where $C={(8\|\phi\|_{\infty}(1 + 3\|\kappa\|_{\infty}\mu(\XX)))^{-1}}$.
	\item Let for $n \ge 1$,  $\Lambda_n \in \clb(\XX)$ be such that
	$$ g(n) := 8  \sup_{\bfx \in \Lambda_n}|\phi(\bfx)|\left(1+ 3 \mu(\XX) \sup_{(\bfx,\bfy) \in \Lambda_n \times \Lambda_n} |\kappa(\bfx,\bfy)| \right) < \infty .$$
	Then for all $m \in \NNN$,
	$$ \prob\{  \vol_{\phi}(\cI^{\sss \RG}_1) > m  \} \le n\mu(\Lambda_n^c)  + 2 nD\exp( - \Delta^2 m/ g(n) ) + \prob(|\PP_n| \ge nD). $$
\end{enumeratei}
\end{Lemma}

{\bf Proof:} Part (i) has been proved in \cite{bhamidi-budhiraja-wang2011} (see Lemmas 6.12 and 6.13 therein).  We now prove (ii). 
Consider the truncated version of $\RG^{\sss(n)}$ constructed using the basic structure $\{ \XX, \bar\mu, \bar \kappa , \bar \phi \}$ where  $\bar \mu: = \mu|_{\Lambda_n}$ (i.e. the restriction of $\mu$ to $\Lambda_n$),   $\bar \kappa(\bfx,\bfy)=\kappa(\bfx,\bfy) \ind_{\Lambda_n}(\bfx)\ind_{\Lambda_n}(\bfy)$ and $\bar \phi(\bfx)=\phi(\bfx)\ind_{\Lambda_n}(\bfx)$.  
Note that $||\bar \kappa||_{\infty} < \infty$ and $||\bar \phi||_{\infty} < \infty$. Denote by $\bar \KK$ the integral operator associated with $(\bar \kappa, \bar \mu)$.  Clearly $\|\bar \KK\| \leq \|\KK\|$ and thus $\bar \Delta=1-\|\bar \KK\| \geq \Delta$. Consider the natural coupling between the truncated and original model by using the vertex set $\bar \PP_n: = \PP_n \cap \Lambda_n$. Write $\bar \cI^{\sss \RG}_1$ for the component with the largest
volume  in the truncated model. Then we have
\begin{align*}
	\prob\{ \vol_{\phi}(\cI^{\sss \RG}_1) > m   \}
	\le& \prob\{ \PP_n \cap \Lambda_n^c \neq \emptyset \} + \prob\{ \PP_n \subset \Lambda_n, \vol_{\phi}(\cI^{\sss \RG}_1) > m   \}\\
	=& \prob\{ \PP_n \cap \Lambda_n^c \neq \emptyset \} + \prob\{  \vol_{\phi}(\bar \cI^{\sss \RG}_1) > m    \}\\
	\le&  (1 -\exp\{ -n \mu(\Lambda_n^c) \})  + 2 nD\exp\{ - \Delta^2 m/ g(n) \} + \prob(|\PP_n| \ge nD), 
\end{align*}
where the last inequality follows from part (i) and the fact that $\Delta \leq \bar \Delta$. \qed

For the proof of the following elementary lemma we refer the reader to Lemma 6.5 in \cite{bhamidi-budhiraja-wang2011}.
\begin{Lemma}
	\label{lemma:basic-operator-norm}
	Let $\kappa, \kappa'$ be kernels on a common finite measure space $(\XX, \mu)$, with the associated integral operators $\KK, \KK'$ respectively. Then
\begin{enumeratea}
	\item $\|\KK\| \le \|\kappa\|_{2,\mu}:=\left( \int_{\XX \times \XX}\kappa^2(\bfx,\bfy) \mu(d\bfx)\mu(d\bfy) \right)^{1/2}$.
	\item If $ \kappa \le \kappa'$, then $||\KK|| \le ||\KK'||$.
	\item $||\KK-\KK'|| \le ||\bar \KK||,$ where $\bar \KK$ is the integral operator associated with $(|\kappa-\kappa'|,\mu)$.
\end{enumeratea}

\end{Lemma}

For the proof of the following lemma we refer the reader to Lemma 6.18 of \cite{bhamidi-budhiraja-wang2011}.
\begin{Lemma}
	\label{lemma:change-of-measure}
	Let $\tilde \mu, \mu$ be two finite measures on the space $\XX$. Assume $\tilde \mu \ll \mu$ and let $g=d\tilde \mu /d\mu$ be the Radon-Nikodym derivative. Let $\tilde \kappa$ be a kernel on $\XX \times \XX$, and define $\kappa$ as 
	$$  \kappa(\bfx,\bfy):=\sqrt{g(\bfx)g(\bfy)} \tilde \kappa(\bfx,\bfy), \;\; \bfx,\bfy\in \XX. $$
	Denote by $\KK$ [resp. $\tilde \KK$] the integral operator on $L^2(\XX, \mu)$ [resp. $L^2(\XX, \tilde \mu)$] associated with $( \kappa, \mu)$ [resp. $( \tilde\kappa, \tilde\mu)$]. Then 
	$ \|\KK\|_{L^2(\mu)}=\|\tilde \KK\|_{L^2(\tilde \mu)} $, where $\|\KK\|_{L^2( \mu)}$ [resp. $\|\tilde \KK\|_{L^2(\tilde \mu)}$] is the norm of the operator $\KK$ [resp. $\tilde \KK$] on $L^2( \mu)$ [resp. $L^2(\tilde\mu)$].
\end{Lemma}

We end this section with a lemma drawing a connection between the Yule process and the pure jump Markov processes with distribution $\nu_{s,i}$ that arose in the construction of the inhomogeneous random graphs $\RG(\mvalpha,\beta,\mvgamma)$, see \eqref{eqn:nu-si-defn}. Fix $j\geq 1$ and recall that a rate one Yule process started at time $t=0$ with $j$ individuals is a pure birth Markov process $Y(t)$ with $Y(0) = j$ and the rate of going from state $i$ to $i+1$ given by $\lambda(i) := i$. Also recall from \eqref{eqn:def-xx} that $\cW$  denotes the Skorohod space $\WW:=\DD([0, \infty): \NNN_0)$.   


\begin{Lemma}
	\label{lemm:yule-domn}
Fix $1\leq i\leq K$ and $s\geq 0$ and rate functions $\mvalpha,\beta,\mvgamma$. Let $\{w(t)\}_{t \ge 0}$ be a pure jump Markov process with law $\nu_{s,i}:= \nu_{s,i}(\mvalpha,\beta,\mvgamma)$ as in \eqref{eqn:nu-si-defn}. Then
\begin{enumeratei}
	\item The process $ w^*(t) := w(t/K \|\mvgamma\|_\infty)/K$ can be stochastically dominated by a Yule processes $Y(\cdot)$ starting with two particles (i.e. $Y(0)=2$).
	\item Fix $t> 0$, $s \in [0,t]$ and $1 \le i \le K$. Then we have
	$$\int_{\WW}[w(t)]^2 \nu_{s,i}(dw) \le  6K^2 e^{2t K \|\mvgamma\|_\infty},$$
	and for any $A > 0$ we have
	$$ \nu_{s,i}( w(t) > A ) \le  2 (1- e^{-t K \|\mvgamma\|_\infty})^{A/2K}. $$
\end{enumeratei}
\end{Lemma}
\textbf{Proof:} Let us first prove (i).  Note that under the law $\nu_{s,i}$, the process $w$ satisfies $w(u) = 0$ for $u< s$ and $w(s) = K+i \leq 2K$. Further for times $t > s$, by \eqref{eqn:attachment-operator}, the jumps of the $w$ can be bounded as  $\Delta w(t):=  w(t)-w(t-) \le K$ at rate at most $w(t) \|\mvgamma\|_\infty$. The process $w^*(\cdot)$ is obtained by rescaling time and space for the process $w(\cdot)$. This is once again a pure jump Markov process  with jump sizes $\Delta w^*(t) \leq 1 $ which happen at rate at most one.  Further $w^*(0) \leq 2$. This immediately implies that this process is stochastically dominated by a Yule process with $Y(0) =2$. This completes the proof.  

We now consider (ii). We will use the result in part (i). Note that a Yule process started with two individuals at time $t=0$ has the same distribution as the sum of two independent Yule processes $\{Y_1(t)\}_{t \ge 0}$ and $\{Y_2(t)\}_{t \ge 0}$ with $Y_1(0)=Y_2(0)=1$. Now fix $t> 0$, $s\leq t$ and $1\leq i\leq K$. Let $w(\cdot)$ have distribution $\nu_{s,i}$.  From (i) we have
\begin{equation}
\label{eqn:wt-bd-2-yule}
	w(t) \le_d  K \cdot (Y_1( t K \|\mvgamma\|_\infty) + Y_2( t K \|\mvgamma\|_\infty) ).
\end{equation}  
For simplicity write $X_1=Y_1( t K \|\mvgamma\|_\infty)$ and $X_2=Y_2( t K \|\mvgamma\|_\infty)$. Well known results about Yule processes (\cite[Chapter 2]{norris-mc}) say that the random variables $X_1$ and $X_2$ have a Geometric distribution with $p:=e^{-tK\|\mvgamma\|_\infty}$.  The first bound in (ii) follows from the Geometric distribution and the fact
\begin{align*}
	\int_{\WW}[w(t)]^2 \nu_{s,i}(dw) \le K^2 \E[(X_1+X_2)^2].
\end{align*}
The second bound in (ii) follows from 
$$ \nu_{s,i}(\{ w(t) > A \}) \le 2 \prob\{X_1 > A/2K\}. $$
This completes the proof. \qed

\subsection{Some perturbation estimates for $\RG(\va,b,\vc)$. }
\label{sec:proof-norm-abc}

Recall that Lemma \ref{lemma:bsr-to-irg} coupled the evolution of $\mirror(t)$ (equivalently $\BS^*(t))$ with two inhomogeneous random graphs $\RG(\va^+,b,\vc^+)(t)$ and $\RG(\va^-,b,\vc^-)(t)$ which can be considered as perturbations of $\RG(\va,b,\vc)(t)$. 
The aim of this section is to understand the effect of such perturbations on the associated operator norms.   Throughout this section $\XX$ and $\phi_t$ are as in \eqref{eqn:def-xx} and \eqref{eqn:def-phi}, respectively.
Given the basic structure $\{\XX, \mu_t, \kappa_t, \phi_t\}$, $t >0$, associated with rate functions $(\mvalpha,\beta,\mvgamma)$, we denote the norm of the integral operator $\KK_t$
associated with $(\kappa_t, \phi_t)$ as $\rho_t(\mvalpha,\beta,\mvgamma)$.

The following proposition
which is the main result of this section studies the effect of perturbations of $(\mvalpha,\beta,\mvgamma)$ on this norm.
For a $K$-dimensional vector $\bfv = (v_1, \cdots v_K)$ and a scalar $\theta$, $\bfv + \theta$ denotes the vector $(v_1 + \theta, \cdots v_K+\theta)$
and $(\bfv +\theta)^+$ denotes the vector $((v_1 + \theta)^+, \cdots (v_K+\theta)^+)$.
\begin{Proposition}
	\label{lemma:error-norm-abc}
	For  $\eps >0 $   let $\rho_t^+=\rho_t(\mvalpha+\eps,\beta+\eps,\mvgamma+\eps)$ and $\rho_t^-=\rho_t((\mvalpha-\eps)^+,(\beta-\eps)^+,(\mvgamma-\eps)^+)$,
	where $(\mvalpha,\beta,\mvgamma)$ are rate functions. 
	Assume that
	$\max\{||\mvalpha +\eps||_{\infty},||\beta +\eps||_{\infty},
	||\mvgamma +\eps||_{\infty}\} < 1.$
	For every $T > 0$, there is a  $C_5 \in (0,\infty)$ such that for all $\eps> 0$ and $t \in [0,T]$,
	$$ \max\{ |\rho_t-\rho_t^+|, |\rho_t-\rho_t^-| \}  \le C_5 \sqrt{\eps} \cdot \left(-\log \eps\right)^2.$$
\end{Proposition}
Proof of Proposition \ref{lemma:error-norm-abc} relies on Lemmas \ref{lemma:disturb-b} -- \ref{lemma:disturb-c} below, and is given at the end of the section.
We  analyze the effect of perturbation of   $\beta$, $\mvalpha$ and  $\mvgamma$ separately in Lemmas  \ref{lemma:disturb-b}, \ref{lemma:disturb-a}) and  \ref{lemma:disturb-c}, respectively.  
%
%
%
%
\begin{Lemma}[Perturbations of $\beta$]
	\label{lemma:disturb-b}
	Let $(\mvalpha,\beta,\mvgamma)$ be rate functions and $\beta^\eps$ be be a nonnegative function on $[0,\infty)$  with $\sup_{0 \le s < \infty}|\beta^\eps(s)-\beta(s)| \le \eps$. Then 
	$$ |\rho_t(\mvalpha,\beta,\mvgamma)-\rho_t(\mvalpha,\beta^\eps,\mvgamma)| \le C \eps,$$
	where $C= 6 t^2 K^3 \|\mvalpha\|_\infty e^{2t\|\mvgamma\|_\infty }$.
\end{Lemma} 
{\bf Proof:} Let $(\mu_t, \kappa_t)$ be the type measure and kernel associated with $(\mvalpha,\beta,\mvgamma)$
Note that a perturbation in $\beta$ only affects the kernel $\kappa_t$ and not  $\mu_t$. 
Recall the representation of $\mu_t$ in terms of probability measures $\{\nu_{u,i}, u \in [0,t], i = 1, \cdots , K\}$.  
From Lemma \ref{lemm:yule-domn}(ii)  
\begin{equation}
 \int_{\WW}[w(t)]^2 \nu_{u,i}(dw) \le  6K^2 e^{2t K\|\mvgamma\|_\infty}, \mbox { for all } u \in [0, t], i = 1, \cdots K. \label{eqn:bound-yule}
\end{equation}
Denote the kernel obtained by replacing $\beta$ by $\beta^{\eps}$ in \eqref{eqn:def-kappa-gen} by $\kappa^{\eps}_t$.
  Since $\|\beta - \beta^\eps\|_\infty < \eps$, we have from \eqref{eqn:def-kappa-gen} that
\[|\kappa_t(\bfx,\bfy)- \kappa^\eps_t(\bfx,\bfy)| \leq \eps \int_0^t w(u){\tilde w}(u) du \leq \eps t w(t){\tilde w}(t), \]
$\mu_t \otimes \mu_t$ a.e.  $(\bfx, \bfy) =((s,w),(r, \tilde w))$.

 By Lemma \ref{lemma:basic-operator-norm} (a) and (c)  we now have  
\begin{align*}
	|\rho_t(\mvalpha,\beta,\mvgamma)-\rho_t(\mvalpha,\beta^\eps,\mvgamma)| 
	\le&  \left( \int_{\XX \times \XX} |\kappa_t(\bfx,\bfy)-\kappa_t^\eps(\bfx,\bfy)|^2 d\mu_t(\bfx)d\mu_t(\bfy) \right)^{1/2}\\
	\le& \left( \int_{\XX \times \XX} (\eps t w(t) \tilde w(t))^2 d\mu_t(\bfx)d\mu_t(\bfy) \right)^{1/2}\\
	=& \eps t \sum_{i=1}^K\int_0^t \alpha_i(s) \left[\int_{\WW} [w(t)]^2 \nu_{s,i}(dw)\right] ds \\
	\le& \eps t \cdot t\|\mvalpha\|_\infty \cdot K \cdot 6K^2 e^{2tK\|\mvgamma\|_\infty},
\end{align*}
where the last inequality follows from \eqref{eqn:bound-yule}.  The result follows. \qed

When $\mvalpha$ or $\mvgamma$ is perturbed, the underlying measure $\mu_t$ changes as well and thus one needs to analyze the corresponding Radon-Nikodym derivatives.
This is done in the  following two lemmas. 
 We denote by $\set{\clg_s}_{0\leq s < \infty}$  the canonical filtration on $\cD([0,\infty):\NNN_0)$. 
In the following we follow the convention that $0/0 = 1$.
\begin{Lemma}
	\label{lemma:rn-derivative}
Fix $\eps> 0$ and let $(\mvalpha,\beta, \mvgamma)$, $(\tilde \mvalpha,  \tilde \beta, \tilde \mvgamma)$ be two sets of rate functions such that for all $1 \le i \le K$ and $s \ge 0$,
	$$\alpha_i(s) - \eps \le \tilde \alpha_i(s) \le \alpha_i(s), \;\mbox{ and } \; \gamma_i(s) - \eps \le \tilde \gamma_i(s) \le \gamma_i(s).$$
	Fix $t \ge 0$ and let $\mu_t$ and $\tilde \mu_t$ be the corresponding type measures on $\XX$. For $(s,w) \in \XX$ and $j\geq 1$ let $\tau_j^s$ for the time of the $j$-th jump of $w(\cdot)$ after time $s$ ($\mu_t$ a.s.).  Also write $\Delta(u)=w(u)-w(u-)$ for $u \ge 0$. Then there exists $\eps_0 > 0$ such that for all $\eps \in (0, \eps_0)$,
	$\tilde \mu \ll \mu$ and 
	$$ \frac{d \tilde \mu_t}{d \mu_t}(s,w)= \frac{\tilde \alpha_{\Delta(s)-K}(s)}{\alpha_{\Delta(s)-K}(s)} \times \Pi_{j : \tau_j^s \le t} \frac{\tilde \gamma_{\Delta(\tau_j^s)}(\tau_j^s)}{\gamma_{\Delta(\tau_j^s)}(\tau_j)} \times \exp \left\{ -\int_s^t w(u) \left[ \sum_{i=1}^K \tilde \gamma_i(u)- \sum_{i=1}^K \gamma_i(u) \right ]du \right\}.  $$ 
\end{Lemma}
\textbf{Proof: } 
For $i=1, \cdots K$, define finite measures $\mu^i_t$, $\tilde \mu^i_t$ on $\XX$ as
$$
\mu^i_t(du\, dw) = \alpha_i(u) \nu_{u,i}(dw)1_{[0,t]}(u) du, \; \tilde \mu^i_t(du\, dw) = \tilde \alpha_i(u) \tilde \nu_{u,i}(dw)1_{[0,t]}(u) du,$$
where $\nu_{u,i}$ is defined above \eqref{eqn:def-mu-gen} and $\tilde \nu_{u,i}$
 is defined similarly on replacing $\gamma_i$ with $\tilde \gamma_i$.  We will show that
\begin{equation}\label{eq:eq707}   \mbox{ for all }  1\le k \le K  \mbox{ and } s \in [0,T],
 \tilde \nu_{s,k} \ll \nu_{s,k} \mbox{ and } \frac{d\tilde \nu_{s,k}}{d\nu_{s,k}}(w) = L_s^t(w),
\end{equation}
where
\begin{equation}
L_s^t : =\Pi_{j \ge 1} \left(\frac{\tilde \gamma_{\Delta(\tau_j^s)}(\tau_j^s)}{\gamma_{\Delta(\tau_j^s)}(\tau_j^s)}
\ind_{\{\tau_j^s \le t\}}\right) \times \exp \left\{ -\int_s^t w(u) \left[ \sum_{i=1}^K \tilde \gamma_i(u)- \sum_{i=1}^K \gamma_i(u) \right ]du \right\}.
\end{equation}
The lemma is an immediate consequence of \eqref{eq:eq707} on observing that
$\mu^i_t$ and $\mu^j_t$ are mutually singular when $i \neq j$, and the relation
$\mu_t = \sum_{i=1}^K \mu^i_t$, $\tilde \mu_t = \sum_{i=1}^K \tilde \mu^i_t$.

We now show \eqref{eq:eq707}.  From the construction of $\nu_{s,k}$ it follows that, there are counting processes $\{N_i(u)\}_{u \in [s,t]}$, $i = 1, \cdots K$, on $\WW$
such that
\be
w(u) = w(s) + \sum_{i=1}^K i N_i(u), \mbox{ for } u \in [s, t], \; a.s.\; \nu_{s,k} \label{eqn:def-counting-jumps}
\ee
and 
\be
M_i(u) := N_i(u) - \int_s^u w(r)\gamma_i(r) dr \;\; \mbox{ under $\nu_{s,k}$ is a } \{\clg_u\}_{u \in [s,t]} \mbox{ local martingale for } \;\; u \in [s,t].  \label{eqn:def-martingale-m}
\ee
%
From standard results it follows that  $L_s^t$ is a local-martingale and super-martingale (see Theorem VI.T2, p.165 of \cite{bremaud1981}).
In order to show \eqref{eq:eq707}, it suffices to show that $\{L_s^u\}_{u \in [s,t]}$ is a martingale. 
By a change of variable formula it follows that (see e.g. Theorem A4.T4, p. 337 of \cite{bremaud1981}) 
\begin{equation}
	L_s^v = 1 + \sum_{i=1}^K \int_s^v L_s^{u-} \cdot \left( \frac{\tilde \gamma_i(u)}{\gamma_i(u)} -1\right)dM_i(u),\; v \in [s,t]. \label{eqn:l-mart-decompose}
\end{equation}
In order to show $L_s^t$ is a martingale, it then suffices, in view of \eqref{eqn:def-martingale-m}, to show that (see e.g. Theorem II.T8 in \cite{bremaud1981}) 
for all $1\le i\le K$,
\[
\int_{\WW} \left[\int_s^t L_s^{u} \cdot |\tilde \gamma_i(u)-\gamma_i(u)| w(u) du \right] d\nu_{s,k}(w) <  \infty .
\]
Finally note that $L_s^{u}  \le e^{ \eps t w(t) } $.
Using Lemma \ref{lemm:yule-domn}(i) and standard estimates for Yule processes, it follows that for $\eps$ sufficiently small
$\sup_{s\in [0,t]} \sup_{1 \le k\le K}\int_{\WW} w(t)e^{ \eps t w(t) }d\nu_{s,k}(w) < \infty$.

The result follows.\qed\\

We will  now use the above lemma to study  the effect of perturbations in $\mvalpha$ on $\rho_t(\mvalpha, \beta, \mvgamma)$.

\begin{Lemma}[Perturbations of $\mvalpha$]
	\label{lemma:disturb-a}
Fix $\eps> 0$.	Let  $(\mvalpha,\beta,\mvgamma)$ be rate functions and let $\mvalpha^\eps = (\alpha_1^{\eps}, \cdots , \alpha_K^{\eps})$, where
$\alpha_i^{\eps}$ are continuous nonnegative bounded functions on $[0,\infty)$ such that for all $1 \le i \le K$ and $s \in [0,T]$
	$$ \alpha_i(s)-\eps \le  \alpha^\eps_i(s) \le \alpha_i(s). $$
	Then for every $t >0$,
	$$ |\rho_t(\mvalpha,\beta,\mvgamma)-\rho_t(\mvalpha^\eps,\beta,\mvgamma)| \le C \sqrt{\eps},$$
	where $C=t \|\beta\|_\infty \cdot 6K^2e^{2t K\|\mvgamma\|_\infty} \cdot 4tK\sqrt{\|\mvalpha\|_\infty}$.
\end{Lemma}
{\bf Proof:} 
Let $(\mu_t,\kappa_t)$ be the type measure and kernel associated with $(\mvalpha,\beta,\mvgamma)$.  Also, let
$\mu^{\eps}_t$ be the type measure associated with $(\mvalpha^{\eps},\beta,\mvgamma)$.
By Lemma \ref{lemma:rn-derivative}, 
$$g(s,w):= \frac{d\mu_t^\eps}{d\mu_t}(s,w)=\frac{ \alpha^\eps_{\Delta(s)-K}(s)}{\alpha_{\Delta(s)-K}(s)} \mbox{ for } (s,w) \in [0,t]\times \WW. $$ 
Using Lemma \ref{lemma:basic-operator-norm} (c), (a), Lemma \ref{lemma:change-of-measure}, and the fact that $|\kappa_t(\bfx,\bfy)| \le t\|\beta\|_\infty  w(t)\tilde w(t)$,
$\mu_t\otimes \mu_t$ a.e.   $(\bfx, \bfy) =((s,w), (\tilde s, \tilde w))$,  we have
\begin{align}
	|\rho_t(\mvalpha,\beta,\mvgamma)-\rho_t(\mvalpha^\eps,\beta,\mvgamma)| 
	\le&  \left( \int_{\XX \times \XX} |\sqrt{g(\bfx)g(\bfy)}-1|^2|\kappa_t(\bfx,\bfy)|^2 d\mu_t(\bfx)d\mu_t(\bfy) \right)^{1/2} \nonumber\\
	\le&  t \|\beta\|_\infty \left( \int_{\XX \times \XX} |\sqrt{g(\bfx)g(\bfy)}-1|^2 w^2(t) \tilde w^2(t) d\mu_t(\bfx)d\mu_t(\bfy) \right)^{1/2} \nonumber\\
	\le&  t \|\beta\|_\infty d_1 \left( \sum_{i,j=1}^K  \int_{[0,t]^2} \left(\sqrt{\frac{\alpha_i^\eps(s)  \alpha_j^\eps(u)}{\alpha_i(s) \alpha_j(u)}}-1\right)^2 \alpha_i(s)\alpha_j(u) dsdu \right)^{1/2}, \label{eqn:651}
\end{align}
where $$d_1 = \sup_{s \in [0,T]}\sup_{1 \le i \le K} \int_{\WW} |w(t)|^2 \nu_{s,i}(dw) \le 6K^2 e^{2tK\|\mvgamma\|_\infty},$$
and the last inequality follows from \eqref{eqn:bound-yule}.
In order to bound \eqref{eqn:651}, note that:
\begin{align*}
	\biggl|\sqrt{\alpha_i(s) \alpha_j(u)}-\sqrt{\alpha_i^\eps(s)  \alpha_j^\eps(u)}\biggr|
	=& \biggl|\sqrt{\alpha_i(s)}\left(\sqrt{\alpha_j(u)}-\sqrt{\alpha_j^\eps(u)}\right)+ \left(\sqrt{\alpha_i(s)}-\sqrt{\alpha_i^\eps(s)}\right)\sqrt{\alpha_j^\eps(u)}\biggr|\\
	\le& 2\sqrt{\eps}\left(\sqrt{\alpha_i(s)}+\sqrt{\alpha_j(u)} \right).
\end{align*}
Plugging the above bound in \eqref{eqn:651} gives the desired result.\qed \\

We will now   analyze the effect of perturbations in $\mvgamma$ on $\rho_t(\mvalpha, \beta, \mvgamma)$.  We need the following preliminary truncation lemma.

\begin{Lemma}
	\label{lemma:norm-trunc-kernel}
For every $T > 0$, there exist $C_6, C_7, A_0 \in (0, \infty)$ such that for any $t\in [0,T]$ and rate functions 	 $(\mvalpha,\beta,\mvgamma)$ the following holds:
Let $\mu_t$, $\kappa_t$ be the type measure and kernel associated with $(\mvalpha,\beta,\mvgamma)$. Define, for $A \in (0, \infty)$, the kernel $\kappa_{A,t}$ as
	\begin{equation}\label{eq:eq250} \kappa_{A,t}(\bfx,\bfy)=\kappa_t(\bfx,\bfy)\ind_{\{ w(t) \le A, \tilde w(t) \le A \}} \text{ where } \bfx=(s,w), \bfy=(r,\tilde w).\end{equation}
Then  for all $A > A_0$,
	$$  \rho(\kappa_t)- C_6 e^{-C_7 A} \le \rho(\kappa_{A,t}) \le \rho(\kappa_t),$$
	where $\rho(\kappa_t)$ [ resp. $\rho(\kappa_{A,t})$] denotes the norm of the operator associated with $(\kappa_t, \mu_t)$ [ resp. $(\kappa_{A,t}, \mu_t)$].
\end{Lemma}
\noindent {\bf Proof:} 
The upper bound in the lemma is immediate from
Lemma \ref{lemma:basic-operator-norm} (b).  We now consider the lower bound.
 For the rest of the proof, we suppress the dependence of $\kappa_t, \kappa_{A,t}, \mu_t$ on $t$.  Note that, from Lemma \ref{lemma:basic-operator-norm} (a,c)
\begin{align}
	\rho(\kappa)-\rho(\kappa_A)
	\le& \left(\int_{\XX\times\XX} (\kappa(\bfx,\bfy)-\kappa_A(\bfx,\bfy))^2 d\mu(\bfx) d\mu(\bfy)\right)^{1/2}\nonumber\\
	\le& 2\left( \int_{\XX\times \XX} (t \|\beta\|_\infty w(t) \tilde w(t)\ind{\{ \tilde w(t) > A\}})^2 d\mu(\bfx) d\mu(\bfy)\right)^{1/2}\nonumber\\
	\le& 2 t \|\beta\|_\infty \left( d_1\sum_{i=1}^K \sum_{j=1}^K\int_{[0,t]\times [0,t]} \alpha_i(s) \alpha_j(u) dsdu  \right)^{1/2}\nonumber\\
	\le& 2 t \|\beta\|_\infty \cdot t \|\mvalpha\|_\infty \cdot \sqrt{d_1}, \label{eq:eq310}
\end{align}
where 
\begin{equation}\label{eq:eq311}
d_1 = \int_{\WW} [w(t)]^2 \nu_{s,i}(dw) \int_{\WW} [ w(t)]^2\ind_{\{ w(t) > A\}} \nu_{s,i}(dw).\end{equation}
By \eqref{eqn:wt-bd-2-yule}, $w(t) \leq_d K(X_1+ X_2)$ where $X_1, X_2$ are independent and identically distributed with Geometric $p=e^{-tK\|\mvgamma\|_\infty}$ distribution.   
\begin{align*}
	\int_{\WW} [ w(t)]^2\ind_{\{ w(t) > A\}} \nu_{s,i}(dw) 
	\le& K^2 \E\left[ (X_1+X_2)^2 \ind_{\set{X_1+X_2> A/K}} \right]\\
	=&  K^2 \E \left[(X_1+ X_2)^2 (\ind_{\set{X_1+X_2 > C, X_1\geq X_2}} + \ind_{\set{X_1+X_2 > C, X_1 <  X_2}})\right]\\
	\le& 4 K^2 \E\left[ X_1^2\ind_{\set{X_1> A/2K}}+ X_2^2\ind_{\set{X_2> A/2K}} \right],
\end{align*}
The above quantity can be bounded by
$$ d_2 (1 - e^{-2TK\|\mvgamma\|_\infty} )^{A/2K} \le d_2 \exp\set{  -\frac{e^{-2TK\|\mvgamma\|_\infty}}{K} A} $$
for some constant $d_2$.
The result now follows on using the above bound and \eqref{eqn:bound-yule} in \eqref{eq:eq311} and \eqref{eq:eq310}.
\qed 


\begin{Lemma}[Perturbations of $\mvgamma$]
	\label{lemma:disturb-c}
	For every $T>0$, there exists $C_8 \in (0, \infty)$ and  $\eps_0 \in (0,1)$ such that for all $t \in [0,T]$ and rate functions
  $(\mvalpha,\beta,\mvgamma)$ the following holds:
 Suppose $\eps \in (0, \eps_0)$ and $\mvgamma^\eps = (\gamma_1^{\eps}, \cdots \gamma_K^{\eps})$,
where $\gamma_i^{\eps}$ are continuous, nonnegative maps on $[0,T]$ such that for all $1 \le i \le K$  
	$$ \gamma_i(s) - \eps \le \ \gamma^\eps_i(s) \le \gamma_i(s), \mbox{ for all } s \in [0,T].$$
	Then 
	$$ |\rho_t(\mvalpha,\beta,\mvgamma)-\rho_t(\mvalpha,\beta,\mvgamma^\eps)| \le C_8 \sqrt{\eps} \cdot (-\log \eps)^2.$$
\end{Lemma}
{\bf Proof:} 
Let $(\mu_t, \kappa_t)$ [resp. $(\mu_t^{\eps}, \kappa_t^{\eps})$] be the type measure and kernel associated with $(\mvalpha,\beta,\mvgamma)$ [resp. $(\mvalpha,\beta,\mvgamma^{\eps})$].
By Lemma \ref{lemma:rn-derivative}, for $(s,w) \in [0,t]\times \WW$
\begin{equation*}
	\frac{d\mu_t^\eps}{d\mu_t} (s,w) = \Pi_{j \ge 1} \left(\frac{ \gamma^\eps_{\Delta(\tau_j^s)}(\tau_j^s)}{\gamma_{\Delta(\tau_j^s)}(\tau_j^s)}\ind_{\{\tau_j^s \le t\}}\right) \times \exp \left\{ -\int_s^t w(u) \left[ \sum_{i=1}^K  \gamma^\eps_i(u)- \sum_{i=1}^K \gamma_i(u) \right ]du \right\}.
\end{equation*}
Denote the right side as $L_s^t(w)$.  Then, as in the proof of Lemma \ref{lemma:rn-derivative}, it follows that
 $\{L_s^u(w)\}_{u \in [s,t]}$ is a $\{\clg_u\}_{u \in [s,t]}$ martingale under $\nu_{s,k}$ for every $k = 1, \cdots K$.
Fix $A\in (A_0, \infty)$, where $A_0$ is as in Lemma \ref{lemma:norm-trunc-kernel}, and let $\kappa_{A,t}$ be defined by \eqref{eq:eq250}.  Similarly define
$\kappa_{A,t}^{\eps}$  by replacing $\kappa_t$ with $\kappa_t^{\eps}$ in \eqref{eq:eq250}.  Denote the operator norm
of the integral operators associated with $(\kappa_{A,t}, \mu_t)$ and $(\kappa_{A,t}^{\eps}, \mu_t^{\eps})$
by $\rho_{A,t}(\mvalpha,\beta,\mvgamma)$ and $\rho_{A,t}(\mvalpha,\beta,\mvgamma^{\eps})$, respectively.  Then, by Lemma 
 \ref{lemma:change-of-measure} and Lemma \ref{lemma:basic-operator-norm} (a,c), 
\begin{align}
	&|\rho_{A,t}(\mvalpha,\beta,\mvgamma)-\rho_{A,t}(\mvalpha,\beta,\mvgamma^\eps)| \nonumber\\
	\le&   \left( \int_{\XX \times \XX} \left|\sqrt{\frac{d\mu_t^\eps}{d\mu_t} (s,w) \frac{d\mu_t^\eps}{d\mu_t} (u,\tilde w)}-1\right|^2 (\kappa_{A,t}(\bfx,\bfy))^2 d\mu_t(\bfx)d\mu_t(\bfy) \right)^{1/2} \nonumber\\
	\le& tA^2\|\beta\|_\infty \left(  \sum_{i=1}^K \sum_{j=1}^K\int_{[0,t]\times [0,t]} \alpha_i(s)\alpha_j(u) 
	\int_{\WW \times \WW} \left| \sqrt{L_s^t(w)L_u^t(w)} -1 \right|^2 \nu_{s,i}(dw)\nu_{u,j}(d\tilde w)\right)^{1/2}. \label{eq:eq415}
\end{align}	
Next, using the martingale property of $L_s^t$, we have
\begin{align}
	&\int_{\WW \times \WW} \left | \sqrt{L_s^t(w)L_u^t(w)} -1 \right |^2 \nu_{s,i}(dw)\nu_{u,j}(d\tilde w)\nonumber \\
	= & 2 - 2	\int_{\WW} \sqrt{L_s^t(w)}\nu_{s,i}(dw) \int_{\WW} \sqrt{L_u^t(w)}\nu_{u,j}(dw) \nonumber\\
	\le & 4- 2\int_{\WW} \sqrt{L_s^t(w)}\nu_{s,i}(dw) - 2\int_{\WW} \sqrt{L_u^t(w)}\nu_{u,j}(dw), \label{eq:eq344}
	\end{align}	
where the inequality on the last line follows on observing that from Jensen's inequality the two integrals on the 	second line are bounded by $1$ and using the elementary
inequality $a_1+a_2 \le a_1a_2+1$, for $a_1,a_2 \in [0,1]$.
	We will now estimate the two integrals on the last line of \eqref{eq:eq344} by using the martingale properties of $\{L_s^u\}_{u \in [s,t]}$ and the representations \eqref{eqn:def-counting-jumps} and \eqref{eqn:l-mart-decompose} in the proof of Lemma \ref{lemma:rn-derivative}.
%
For the rest of the proof we write $L_s^u$  as $L_s(u)$. By an application of Ito's formula, we have that for every $k = 1, \cdots K$, $\nu_{s,k}$ a.s.
\begin{equation}\label{eq:eq405}
\begin{aligned}
	&\sqrt{L_s(t)}-1 
	- \sum_{i=1}^K \int_s^t \frac{\sqrt{L_s(u-)}}{2}\left( \frac{\gamma^\eps_i(u)}{\gamma_i(u)} -1\right)dM_i(u) \\
	=& \sum_{ s < u \le t}\left( \sqrt{L_s(u)}-\sqrt{L_s(u-)} \right) - \sum_{i=1}^K \int_s^t \frac{\sqrt{L_s(u-)}}{2}\left( \frac{\gamma^\eps_i(u)}{\gamma_i(u)} -1\right)dN_i(u)\\
	=& \sum_{i=1}^K \int_s^t \sqrt{L_s(u-)} \left( \sqrt{\frac{\gamma^\eps_i(u)}{\gamma_i(u)}} -1\right)dN_i(u) - \sum_{i=1}^K \int_s^t \frac{\sqrt{L_s(u-)}}{2}\left( \frac{\gamma^\eps_i(u)}{\gamma_i(u)} -1\right)dN_i(u)\\
	=& - \frac{1}{2} \sum_{i=1}^K \int_s^t {\sqrt{L_s(u-)}}\left( \sqrt{\frac{\gamma^\eps_i(u)}{\gamma_i(u)}} -1\right)^2dN_i(u),\\
\end{aligned}
\end{equation}
where the second equality follows on observing that for $u \in (s,t]$
$\sqrt{L_s(u)} = \sum_{i=1}^K \sqrt{L_s(u-)}  \sqrt{\frac{\gamma^\eps_i(u)}{\gamma_i(u)}} \Delta N_i(u)$.

As in the proof of Lemma \ref{lemma:rn-derivative}, we can check that
for all $i,k$,
\[
\int_{\WW} \left[\int_s^t \sqrt{L_s({u}}) \cdot | \gamma_i^{\eps}(u)-\gamma_i(u)| w(u) du \right] d\nu_{s,k}(w) <  \infty,
\]
and consequently the last term on the left side of \eqref{eq:eq405} is a martingale.
Denoting the expectation operator corresponding to the probability measure $\nu_{s,k}$ on $\WW$ by $\E_{s,k}$, we have
\begin{align*}
	1-\E_{s,k}[\sqrt{L_s(t)}] 
	=& \frac{1}{2} \sum_{i=1}^K \E_{s,k} \left[ \int_s^t {\sqrt{L_s(u-)}}\left( \sqrt{\frac{\gamma^\eps_i(u)}{\gamma_i(u)}} -1\right)^2dN_i(u)\right] \\
	=& \frac{1}{2} \sum_{i=1}^K \E_{s,k} \left[ \int_s^t {\sqrt{L_s(u)}}\left( \sqrt{\frac{\gamma^\eps_i(u)}{\gamma_i(u)}} -1\right)^2 w(u) \gamma_i(u) du\right]\\
	=& \frac{1}{2} \sum_{i=1}^K \int_s^t \E_{s,k} \left[ \sqrt{ L_s(u)} w(u)  \left(  \sqrt {\gamma^\eps_i(u)}- \sqrt {\gamma_i(u)} \right)^2 \right] du\\
	\le& \frac{1}{2} \int_s^t K \eps  \cdot \E_{s,k} [ \sqrt{ L_s(u)} w(u) ] du\\
	\le& \frac{K \eps}{2} \int_s^t \left(\E_{s,k} [  L_s(u) ] \E_{s,k} [  w^2(u) ]\right)^{1/2} du\\
	\le& \frac{K \eps}{2} \cdot t \cdot (6K^2 e^{2TK\|\mvgamma\|_\infty})^{1/2},
\end{align*}
where
the last inequality follows from  \eqref{eqn:bound-yule}. 
Using the above bound in \eqref{eq:eq415} we now have
\begin{align*}
	|\rho_{A,t}(\mvalpha,\beta,\mvgamma)-\rho_{A,t}(\mvalpha,\beta,\mvgamma^\eps)|
	 \le & 
	tA^2\|\beta\|_\infty \cdot t\|\mvalpha\|_\infty \cdot \left[2K\eps t (6K^2 e^{2TK||\mvgamma||_{\infty}})^{1/2}\right]^{1/2}.
\end{align*}
Finally, by Lemma \ref{lemma:norm-trunc-kernel}, we have
\begin{align*}
	|\rho_{t}(\mvalpha,\beta,\mvgamma)-\rho_{t}(\mvalpha,\beta,\mvgamma^\eps)| \le& |\rho_{A,t}(\mvalpha,\beta,\mvgamma)-\rho_{A,t}(\mvalpha,\beta,\mvgamma^\eps)| + 2 C_6 e^{-C_7 A}\\
	 < & d_1 A^2 \eps^{1/2}+ 2 C_6 e^{-C_7 A},
\end{align*}
where $d_1=tA^2\|\beta\|_\infty \cdot t\|\mvalpha\|_\infty \cdot \left[2K t (6K^2 e^{2TK||\mvgamma||_{\infty}})^{1/2}\right]^{1/2}$. 
  The result now follows on taking $A=- \log \eps$ in the above display and taking  $\eps_0$ sufficiently small (in particular such that $-\log(\eps_0) > A_0$). \qed\\

Now we combine all the above ingredients to complete the proof of Proposition \ref{lemma:error-norm-abc}.\\

\textbf{Proof of Proposition \ref{lemma:error-norm-abc}:} 
Using Lemma \ref{lemma:disturb-c}, \ref{lemma:disturb-b} and \ref{lemma:disturb-a},   we get 
\begin{align*}
	|\rho_t^+ -\rho_t| 
	\le& |\rho_t(\mvalpha+\eps, \beta+\eps, \mvgamma+\eps)-\rho_t(\mvalpha+\eps, \beta+\eps, \mvgamma)|\\
	&+ |\rho_t(\mvalpha+\eps, \beta+\eps, \mvgamma)-\rho_t(\mvalpha+\eps, \beta, \mvgamma)|+|\rho_t(\mvalpha+\eps, \beta, \mvgamma)-\rho_t(\mvalpha, \beta, \mvgamma)|\\
	\le& C_8 \eps^{1/2}(-\log \eps)^2 + d_1\eps + d_2 \eps^{1/2},
\end{align*}
where
$d_1 = 6 T^2 K^3  e^{2TK }$ and $d_2 =   24K^3T^2e^{2TK}$.
	A similar bound holds for $|\rho_t^- -\rho_t|$. The result follows.\qed

\subsection{Effect of time perturbation  on $\rho_t$}
\label{sec:proof-norm-t}
Throughout this section we fix rate functions $(\mvalpha,\beta,\mvgamma)$.
The aim of this section is to understand the evolution of the operator norm $\rho_t(\mvalpha,\beta,\mvgamma)$  as $t$ changes. The main result of the section is Proposition \ref{lemma:error-norm-t} which studies continuity and differentiability properties of the function 
$\rho(t):= \rho_t(\mvalpha, \beta,\mvgamma)$, $t\geq 0$.


\begin{Proposition}
	\label{lemma:error-norm-t}
	Suppose that  $\beta(t)>0$ for $t>0$ and $\liminf_{t \to \infty}\beta(t) > 0$. Then 
	\begin{enumeratei}
		\item $\rho$ is a continuous strictly increasing function on $\Rbold_+$ with $\rho(0) = 0$ and $\lim_{t\to \infty}\rho(t)=\infty$.
		\item 
		There is a unique value $t_c^\prime = t_c^\prime(\mvalpha,\beta,\mvgamma)$ such that $\rho(t_c^\prime) = 1$. 
			\end{enumeratei}
\end{Proposition}
The proof of the proposition relies on the following lemma and is given after the proof of the lemma.
\begin{Lemma}
	\label{lemma:der-norm-t}
Let $0 < t_1 \le t_2 < \infty$. Then 
\begin{equation*}
	|t_2-t_1| \cdot \frac{\inf_{t_1 \le u \le t_2}\beta(u)}{t_1\|\beta\|_\infty} \cdot \rho(t_1) \le \rho(t_2)-\rho(t_1) \le |t_2-t_1| \cdot 6 t_2 K^2 \|\beta\|_\infty \|\mvalpha\|_\infty e^{2t_2 K\|\mvgamma\|_\infty}.
\end{equation*}
\end{Lemma}
\noindent{\bf Proof:} Letting $\mu:=\mu_{t_2}$ we have 
\begin{align*}
	|\rho(t_2)-\rho(t_1)| 
	\le& \left (\int_{\XX \times \XX} (\kappa_{t_2}(\bfx,\bfy)-\kappa_{t_1}(\bfx,\bfy))^2 d\mu(\bfx) d\mu(\bfy)\right)^{1/2}\\
	\le& \left (\int_{\XX \times \XX} (\|\beta\|_\infty w(t_2) \tilde w(t_2) |t_2-t_1| )^2 d\mu(\bfx) d\mu(\bfy)\right)^{1/2}\\
	\le& |t_2-t_1| \cdot \|\beta\|_\infty \cdot t_2\|\mvalpha\|_\infty \cdot 6K^2e^{2t_2 K\|\mvgamma\|_\infty},
\end{align*}
where the last inequality  once again follows from \eqref{eqn:bound-yule}.
  This proves the upper bound. 

Next note that, for $\mu\otimes \mu$ a.e. $(\bfx, \bfy)$ such that  $\kappa_{t_1}(\bfx,\bfy) \ne 0$, we have
\begin{align*}
	\frac{\kappa_{t_2}(\bfx,\bfy)}{\kappa_{t_1}(\bfx,\bfy)} 
	=& 1 + \frac{\int_{t_1}^{t_2} w(u) \tilde w(u) \beta(u) du}{\int_{0}^{t_1} w(u) \tilde w(u) \beta(u) du}\\
	\ge& 1 + \frac{w(t_1) \tilde w(t_1) \inf_{t_1 \le u \le t_2}\beta(u) \cdot (t_2-t_1)}{ w(t_1) \tilde w(t_1) \|\beta\|_\infty t_1}.
\end{align*}
Thus $\kappa_{t_2}(\bfx,\bfy) \ge \left(1 + |t_2-t_1| \cdot \frac{\inf_{t_1 \le u \le t_2}\beta(u)}{t_1\|\beta\|_\infty}\right) \kappa_{t_1}(\bfx,\bfy)$ which from Lemma \ref{lemma:basic-operator-norm} (b) implies
\begin{equation*}
	\rho(t_2) - \rho(t_1) \ge |t_2-t_1| \cdot \frac{\inf_{t_1 \le u \le t_2}\beta(u)}{t_1\|\beta\|_\infty} \cdot \rho(t_1).
\end{equation*}
This completes the proof of the lower bound. \qed

\textbf{Proof of Proposition \ref{lemma:error-norm-t}:}  
Since $\kappa_0 = 0$, the property $\rho(0) =0$ is immediate.  Also Lemma \ref{lemma:der-norm-t} shows that $\rho$ is continuous and strictly increasing.
Finally since $\inf_{t\to \infty} \beta(t) > 0$, there exists $\delta > 0$ and a $t^* \in (0, \infty)$ such that for all $t \ge t^*$, $\beta(t) \ge \delta$.
From Lemma \ref{lemma:der-norm-t} we then have, for $t \ge t^*$,
$\rho(t) - \rho(t^*) \ge \frac{(t-t^*) \delta}{t^* ||\beta||_{\infty}}$.  This proves that $\rho(t) \to \infty$ as $t\to \infty$
and completes the proof of (i).  Part (ii) is immediate from (i).
%
%
%
%
%
\qed


\subsection{Operator norm of $\RG(\va,b,\vc)$ and critical time of $\BS$}
\label{sec:proof-phase-transition}
In this section we will prove Theorem \ref{theo:new-crit-time}.
Recall that by Lemma \ref{lemma:bsr-to-irg}, for any fixed time $t$, $\BS^*(t)$ (more precisely, $\mirror(t)$) can be approximated by perturbations of $\RG(\va,b,\vc)(t)$. To estimate the volume of the largest component in  $\RG(\va,b,\vc)(t)$ we will use Lemma \ref{lemma:irg-unbounded}.
In order to identify suitable $\Lambda_n$ as in part (ii) of the lemma, we start with the following lemma.

\begin{Lemma}
\label{lemma:error-wst}
Let $(\mvalpha,\beta,\mvgamma)$ be rate functions and let $\mu_t$ be the associated type measure.  Fix $T>0$. Define
$\Lambda \in \clb(\XX)$ as $\Lambda = \{(s,w) \in \XX: w(T) \le l\}$ for $l \in \Rbold_+$.  Then, for every $l \in \Rbold_+$
\begin{equation*}
\mu_t( \Lambda^c )<  2T\|\mvalpha\|_\infty  \cdot  \exp\left(-l\frac{e^{-TK \|\mvgamma\|_\infty }}{2K}\right).
\end{equation*}
\end{Lemma}
\textbf{Proof:} 
Note that
\begin{equation}\label{eq:eq1335}
\mu_t( \Lambda^c ) = \sum_{i=1}^K \int_0^t \alpha_i(u) \nu_{u,i}( \Lambda^c ) \le ||\mvalpha||_{\infty} T \sup_{u \in [0,T]} \sup_{1 \le i \le K}
\nu_{u,i}( \Lambda^c ).\end{equation}
By \eqref{eqn:wt-bd-2-yule}, 
$$
\nu_{u,i}( \Lambda^c ) =\nu_{u,i}(\{ w: w(T) \ge l \}) \le \prob(X_1 +X_2 \ge l/K) \le 2 (1-e^{-TK\|\mvgamma\|_\infty })^{l/2K}.$$
where $X_i$ are iid with  $\mbox{Geom}(e^{-T\|\mvgamma\|_\infty })$ distribution. 
Using this estimate in \eqref{eq:eq1335}, we have
$$
\mu_t( \Lambda^c ) \le ||\mvalpha||_{\infty} T \cdot 2 (1-e^{-TK\|\mvgamma\|_\infty })^{l/2K}.$$
The result follows.
\qed \\ \ \\

We will now use the above lemma along with Lemma \ref{lemma:irg-unbounded} to estimate the largest component in 
$\RG^{\sss(n)}(\mvalpha,\beta,\mvgamma)(t)$.  Recall the notation $\rho_t(\mvalpha,\beta,\mvgamma)$ from Section \ref{sec:proof-norm-abc}.
\begin{Lemma}
	\label{lem:lem1426}
Let $(\mvalpha,\beta,\mvgamma)$ be rate functions  and denote by
 $\cI_1^{\sss \RG}(t)$  the component with the largest volume, with respect to the weight function $\phi_t$, in $\RG^{\sss(n)}(t) := \RG^{\sss(n)}(\mvalpha,\beta,\mvgamma)(t)$. 
Then, for every $t>0$ such that  $\rho_t(\mvalpha,\beta,\mvgamma) < 1$, there exists $A \in (0, \infty)$ such that 
	\[\prob(\vol_{\phi_t}(\cI_1^{\sss \RG}(t)) > A\log^4{n}) \to 0, \mbox{ as } n\to \infty.\] 
\end{Lemma}
\noindent{\bf Proof:} We will use Lemma \ref{lemma:irg-unbounded}(ii). Define
\[\Lambda_n:= \set{(s,w)\in \cX: w(t) < B \log{n}},\] 
where   $B$ will be chosen appropriately later in the proof.   Now consider the function $g(n)$ in Lemma \ref{lemma:irg-unbounded}(ii)
with $\Lambda_n$ defined as above and $(\mu, \phi, \kappa)$ there replaced by
$(\mu_t, \phi_t, \kappa_t)$, where $(\mu_t, \kappa_t)$ is the type measure and kernel associated with $(\mvalpha,\beta,\mvgamma)$.
Note that
 \[\kappa_t(\bfx,\bfy) = \int_0^t \beta(u) w(u) \tilde w(u) du \leq t\|\beta\|_\infty w(t) \tilde w(t)\] 
and therefore 
\begin{equation}
\label{eqn:gn-final-bd}
	g(n) \le 8 B\log{n}(1+ 3\mu_t(\cX) \cdot t\|\beta\|_\infty B^2\log^2{n}).
\end{equation} 
Writing $m_n = A\log^4{n}$,  the bound in Lemma \ref{lemma:irg-unbounded}(ii) then gives
\begin{equation}
\label{eqn:tot-prob-bound}
\prob(\vol_{\phi_t}(\cI_1^{\sss \RG}(t)) > m_n) \leq n\mu_t(\Lambda_n^c) +  2n  \mu_t(\cX)\exp\biggl(-\Delta^2 A \log^4{n}/g(n) \biggr),	
\end{equation}
where $\Delta = 1 - \rho_t(\mvalpha,\beta,\mvgamma) > 0$.
Using Lemma \ref{lemma:error-wst} with $l= B\log{n}$ gives 
\begin{equation}
	n\mu_t(\Lambda_n^c) \leq n t \|\mvalpha\|_\infty \cdot n^{-B e^{-T\|\mvgamma\|_\infty} /2K} = o(1) \label{eqn:1450}
\end{equation} 
for $B >  2Ke^{T||\mvgamma||_\infty}$. 
Now fix $B>  e^{T\|\mvgamma\|_\infty /2K}$, and choose $A$ large such that  
\[n \mu_t(\cX)\exp\biggl(-\Delta^2 A \log^4{n}/g(n) \biggr) \to 0\]
as $n\to\infty$. The result follows. \qed


\textbf{Proof of Theorem \ref{theo:new-crit-time}:} 
Let, for $t \ge 0$, $(\mu_t, \kappa_t)$ be the type measure and the kernel associated with rate functions $(\va,b,\vc)$.  We will prove Theorem \ref{theo:new-crit-time} with this choice of $(\mu_t, \kappa_t)$.  From Proposition \ref{lemma:error-norm-t} we have that
$\rho(t) = \rho_t(\va,b,\vc)$ is continuous and strictly increasing in $t$ and there is a unique $t_c' \in (0, \infty)$ such that
$\rho(t_c') =1$.  It now suffices to show that: (a) For $t < t_c'$, $|\clc_1(t)|$ (the size of the largest component in $\BS^*(t)$) is $O(\log^4 n)$;
and (b) for $t > t_c'$, $|\clc_1(t)| = \Omega(n)$.

Consider first (a).   Fix $t< t_c^\prime$.
For $\delta > 0$, define rate functions $(\va^+,b^+,\vc^+) = (\va +\delta,b+ \delta,\vc +\delta)$.  Since $\rho(t) <1$, by Proposition 
\ref{lemma:error-norm-abc}, we can choose $\delta$ sufficiently small so that $\rho_t(\va^+,b^+,\vc^+) <1$. Denote $\cI_1^{\sss \RG^+}(t)$ for the component of the largest volume in $\RG^+(t) := \RG(\va^+,b^+,\vc^+)(t)$. From Lemma \ref{lem:lem1426}
there exists $A \in (0, \infty)$ such that
$$\prob (\vol_{\phi_t}(\cI_1^{\sss \RG^+}(t)) > A \log^4 n) \to 0, \mbox{ as } n \to \infty .
$$
Combining this result with Corollary \ref{cor:cor2127}  we see that
$$\prob (\vol_{\phi_t}(\cI_1^{\sss \mirror}(t)) > A \log^4 n) \to 0, \mbox{ as } n \to \infty,
$$
where $\cI_1^{\sss \mirror}(t)$ is the component with the largest volume in $\mirror(t)$.  Part (a) is now immediate from the one to one correspondence between
the components in $\mirror(t)$ and $\BS^*(t)$ (see \eqref{eq:eq1848}).

We now consider (b). Fix $t> t_c^\prime$.  Then
 $\rho(t) > 1$. From Proposition \ref{lemma:error-norm-abc} we can find 
$\delta > 0$ such that $\rho_t(\va^-,b^-,\vc^-) > 1$, where
$(\va^-,b^-,\vc^-) = ((\va-\delta)^+, (b-\delta)^+, (\vc-\delta)^+)$.
Let $\CC_1^{\sss \RG^-}(t)$ be the component in $\RG^{-}(t):=\RG^{\sss(n)}(\va^-,b^-,\vc^-)(t)$ with the largest number of vertices. By Theorem 3.1 of \cite{bollobas-riordan-janson}, $|\CC_1^{\sss \RG^-}(t)| = \Theta(n)$.
Since $\vol_{\phi_t}(\CC_1^{\sss \RG^-}(t)) \ge |\CC_1^{\sss \RG^-}(t)|$, we have
$\vol_{\phi_t}(\cI_1^{\sss \RG^-}(t)) = \Omega(n)$, where $\cI_1^{\sss \RG^-}(t)$ is the component with the largest volume in $\RG^{-}(t)$.
Finally, in view of Corollary \ref{cor:cor2127}, we have the same result with $\cI_1^{\sss \RG^-}(t)$ replaced by $\cI_1^{\sss \mirror}(t)$ and the result follows once more from the 
one to one correspondence between
the components in $\mirror(t)$ and $\BS^*(t)$. \qed
%
%
%

\subsection{Barely subcritical regime for Bounded-size rules}
\label{sec:ia-largest-com}
In this section we  complete the proof of Theorem \ref{thm:subcrit-reg}. Throughout this section we fix $\gamma \in (0,1/4)$ and let $t_n = t_c - n^{-\gamma}$. 
The main ingredient in the proof is the  following proposition.
\begin{Proposition}
	\label{prop:size-first-com}
	 There exist  $\bar B,\bar C,\bar N \in (0, \infty)$ such that for all $n \ge \bar N$ and  all  $0 \le t \le t_n$ 
	$$ \prob \{ |\CC_1^{\sss (n)}(t)| \ge \bar m(n,t) \} \le \frac{\bar C}{n^2}, \mbox{ where } \bar m(n,t) =  \frac{\bar B (\log n)^4}{(t_c-t)^2} .$$
\end{Proposition}

Let us first prove Theorem \ref{thm:subcrit-reg} assuming the above proposition.\\

\noindent {\bf Proof of Theorem \ref{thm:subcrit-reg}:}
Write $\tau = \inf\{ t \ge 0: |\CC_1^{\sss (n)}(t)| \ge m(n,t)  \}$, where $m(n,t) =  \frac{2\bar B (\log n)^4}{(t_c-t)^2}$. Then
\begin{align}\label{eq:eq1241}
	\prob\{ |\CC_1^{\sss (n)}(t)| \ge m(n,t) \mbox{ for some } t \le t_n \}=  \prob\{ \tau \le t_n \}.
\end{align}
Note that 
\begin{equation}\label{eq:eq1243}
	\{ \tau =t \} \subset \bigcup_{v,v' \in [n], v\neq v'} E^{v,v'},
\end{equation}
where, denoting the component in $\BS(t)$ that contains the vertex $v \in [n]$ by $\CC_{v}^{\sss (n)}(t)$ and its size by $|\CC_{v}^{\sss (n)}(t)|$,
\begin{align}
E^{v,v'} =& \left\{\max\set{|\CC_{v}^{\sss (n)}(t-)|, |\CC_{v'}^{\sss (n)}(t-)|} < m(n,t); \CC_{v}^{\sss (n)}(t-) \neq \CC_{v'}^{\sss (n)}(t-) \right\}\nonumber\\
&\;\;\;\bigcap \left\{|\CC_{v}^{\sss (n)}(t-)| + |\CC_{v'}^{\sss (n)}(t-)| \ge m(n,t)\right\}
\bigcap \{\CC_{v}^{\sss (n)}(t) = \CC_{v'}^{\sss (n)}(t)\}.	\label{eq:eq1247}
	\end{align}
Note  that
\begin{equation}\label{eq:eq1255}
\prob\{|\CC_{v}^{\sss (n)}(t)| + |\CC_{v'}^{\sss (n)}(t)| \ge m(n,t)\} \le 2 \prob\{  |\CC^{\sss (n)}_1(t)|  \ge m(n,t)/2\}
\end{equation}
and, on the set, $\{\max\set{|\CC_{v}^{\sss (n)}(t)|, |\CC_{v'}^{\sss (n)}(t)|} < m(n,t)\}$, the rate at which
$\CC_{v}^{\sss (n)}(t)$ and $\CC_{v'}^{\sss (n)}(t)$ merge can be bounded by
$$\frac{1}{2n^3} \cdot 4 |\CC_{v}^{\sss (n)}(t)|  | \CC_{v'}^{\sss (n)}(t)| n^2  \le \frac{2 m^2(n,t)}{n}.$$
Combining this observation with \eqref{eq:eq1243} and \eqref{eq:eq1255}, we have
%
\begin{align*}
	\prob\{ \tau \le t_n \} \le& \sum_{v,v' \in [n], v\neq v'}\int_0^{t_n} \prob\{|\CC_{v}^{\sss (n)}(t)| + |\CC_{v'}^{\sss (n)}(t)| \ge m(n,t)\}\cdot \frac{2 m^2(n,t)}{n} dt\\
	\le& 2n^2 \int_0^{t_n}  \prob\{  |\CC^{\sss (n)}_1(t)|  \ge m(n,t)/2\} \cdot \frac{2 m^2(n,t)}{n} dt\\
	\le& 4n t_c \sup_{t \le t_n} \set{m^2(n,t)\prob\{  |\CC^{\sss (n)}_1(t)|  \ge \bar m(n,t)\} }\\
	=& O(n \cdot n^{4\gamma} (\log n)^8 \cdot n^{-2}) = O(n^{-1+4\gamma}(\log n)^8 )=o(1),
\end{align*}
where the last line follows from Proposition \ref{prop:size-first-com} and the fact that $\gamma < 1/4$. 
Using the above estimate in \eqref{eq:eq1241} we have the result.
 \qed\\

We will need the following lemma in the proof of  Proposition \ref{prop:size-first-com}.\\

\begin{Lemma}
	\label{lemma:norm-ia}
Let $(\va^+,b^+,\vc^+) = (\va +\delta_n, b + \delta_n, \vc + \delta_n)$, where
$\delta_n = n^{-2\gamma_0}$ and $\gamma_0 \in (\gamma , 1/4)$.
 Let $\rho_t^{\sss (n), +} = \rho_t(\va^+,b^+,\vc^+)$.   Then there exists  $C_9, N_0 \in (0,\infty)$ such that for all   $n \ge N_0$,
	$$ \rho_t^{\sss (n),+} < 1- C_9 (t_c-t) \text{ for all }  0 \le t \le t_n. $$
\end{Lemma}
\noindent\textbf{Proof of Lemma \ref{lemma:norm-ia}: } From Proposition \ref{lemma:error-norm-abc}, there is a $d_1 \in (0, \infty)$ 
such that
$$ \rho_t^{\sss (n), +}  \le \rho_t(\va,b,\vc) + d_1 n^{-\gamma_0}  \log^2 n, \mbox { for all } t \le t_c. $$
By Lemma \ref{lemma:der-norm-t} and since $\rho_{t_c}(\va,b,\vc)=1$, there exists  $d_2 \in (0, \infty)$ such that
$$ \rho_t(\va,b,\vc) \le  1 - d_2 (t_c-t), \mbox{ for all } t \le t_n.$$
Thus, since $\gamma < \gamma_0$, we have for some $N_0 > 0$
\begin{align*}	
	\rho_t^{\sss (n), +} \le 1 - d_2 (t_c-t)  + d_1 n^{-\gamma_0} ( \log n)^2 	< 1 - \frac{d_2}{2}(t_c-t),
\end{align*}
for all $n \ge N_0$ and $0 \le t \le t_c-n^{-\gamma}$. The result follows. \qed \\

\textbf{Proof of Proposition \ref{prop:size-first-com}:} Recall the rate functions $(\va,b,\vc)$ introduced in Section \ref{sec:evol-bst-st}. 
Choose $\gamma_0\in (\gamma,1/4)$ and let    $(\va^+,b^+,\vc^+)$ be as in Lemma \ref{lemma:norm-ia}.
Fix $t< t_n$ and consider the random graph $\RG^{\sss(n)}(\va^+,b^+, \vc^+)(t)$. From Lemma \ref{lemma:bsr-to-irg},  we can couple $\mirror(t)$ and $\RG^{\sss(n)}(\va^+,b^+, \vc^+)(t)$ such that 
\[\prob(\mirror(t) \subseteq \RG^{\sss(n)}(\va^+,b^+, \vc^+)(t) ) \geq 1-C_3\exp(-C_4 n^{1-4\gamma_0}), \mbox{ for all } t \in [0,T].\]
Recalling the one to one correspondence between components in $\BS^*(t)$ and $\Gamma(t)$, and \eqref{eq:eq1848}, we have
 for any $m\geq 1$, 
\be
 \prob\{ |\CC_1^{\sss (n)}(t)| > m \} \le \prob\{ \vol_{\phi_t}(\cI_1^{\sss \RG^+}(t)) \ge m \} + C_3 \exp\{ - C_4 n^{1-4\gamma_0}\}, \label{eqn:795}
\ee
where $\cI_1^{\sss \RG^+}(t)$ is the component with the largest volume in $\RG^{\sss(n)}(\va^+,b^+, \vc^+)(t)$.
From Lemma \ref{lemma:norm-ia}, there is a $N_0 > 0$ such that
 $\Delta_t^{\sss(n), +} = 1-\rho_t(\va^+, b^+, \vc^+)$ satisfies
\begin{equation}
\label{eqn:delta-diff-1}
	\Delta_t^{\sss(n), +} \geq C_9(t_c-t), \mbox{ for all } t \le t_n, \; n \ge N_0.
\end{equation}
  Using Lemma \ref{lemma:irg-unbounded} and arguing as in equation \eqref{eqn:tot-prob-bound} we have for all $ t \in [0,T] $ and all $m\geq 1$,
\begin{equation}
	\label{eqn:801}
	\prob\{ \vol_{\phi_t}(\cI_1^{\sss \RG^+}(t)) \ge m \} \leq n d_1 \exp\{ - (\Delta^{\sss (n),+}_t)^2 m/ (d_2 \log^3 n )\}  + d_3 n^{-2},
\end{equation}
 where $d_1, d_2, d_3$ are suitable constants. Using \eqref{eqn:delta-diff-1} in \eqref{eqn:801} we get  
\begin{equation*}
	\prob\{ \vol_{\phi_t}(\cI_1^{\sss \RG^+}(t))  \ge m \} \le  nd_1 \exp\{ - d_4 (t_c-t)^2 m/  \log^3 n \}  + d_3 n^{-2}.
\end{equation*}
The result now follows on substituting $ m = m(n,t) =  \frac{\bar B (\log n)^4}{(t_c-t)^2}$, with $\bar B > 3/d_4$, in the above inequality.
\qed\\


\section*{Acknowledgements}

AB and XW has been supported in part by the National Science Foundation (DMS-1004418, DMS-1016441), the Army Research Office (W911NF-0-1-0080, W911NF-10-1-0158) and the US-Israel Binational Science Foundation (2008466). SB and XW have been supported in part NSF-DMS grant 1105581.




\begin{bibdiv}
\begin{biblist}

\bib{achlioptas2009explosive}{article}{
      author={Achlioptas, D.},
      author={D'Souza, R.M.},
      author={Spencer, J.},
       title={{Explosive percolation in random networks}},
        date={2009},
        ISSN={0036-8075},
     journal={Science},
      volume={323},
      number={5920},
       pages={1453},
}

\bib{aldous1997brownian}{article}{
      author={Aldous, D.},
       title={{Brownian excursions, critical random graphs and the
  multiplicative coalescent}},
        date={1997},
     journal={The Annals of Probability},
      volume={25},
      number={2},
       pages={812\ndash 854},
}

\bib{bhamidi-budhiraja-wang2011}{article}{
      author={Bhamidi, S.},
      author={Budhiraja, A.},
      author={Wang, X.},
       title={Aggregation models with limited choice and the multiplicative
  coalescent},
        date={2012},
     journal={To appear in Random Structures and Algorithms},
}

\bib{amc-2012}{article}{
      author={Bhamidi, S},
      author={Budhiraja, A},
      author={Wang, X},
       title={{The augmented multiplicative coalescent and emergence of the
  giant and complexity for dynamic random graph models at criticality}},
        date={2012},
     journal={Arxiv preprint},
}

\bib{bohman2001avoiding}{article}{
      author={Bohman, T.},
      author={Frieze, A.},
       title={{Avoiding a giant component}},
        date={2001},
     journal={Random Structures and Algorithms},
      volume={19},
      number={1},
       pages={75\ndash 85},
}

\bib{bohman2004avoidance}{article}{
      author={Bohman, T.},
      author={Frieze, A.},
      author={Wormald, N.C.},
       title={{Avoidance of a giant component in half the edge set of a random
  graph}},
        date={2004},
        ISSN={1098-2418},
     journal={Random Structures \& Algorithms},
      volume={25},
      number={4},
       pages={432\ndash 449},
}

\bib{bollobas-rg-book}{book}{
      author={Bollob{\'a}s, B{\'e}la},
       title={Random graphs},
     edition={Second},
      series={Cambridge Studies in Advanced Mathematics},
   publisher={Cambridge University Press},
     address={Cambridge},
        date={2001},
      volume={73},
        ISBN={0-521-80920-7; 0-521-79722-5},
      review={\MR{1864966 (2002j:05132)}},
}

\bib{bollobas-riordan-janson}{article}{
      author={Bollob{\'a}s, B{\'e}la},
      author={Janson, Svante},
      author={Riordan, Oliver},
       title={The phase transition in inhomogeneous random graphs},
        date={2007},
        ISSN={1042-9832},
     journal={Random Structures Algorithms},
      volume={31},
      number={1},
       pages={3\ndash 122},
         url={http://dx.doi.org/10.1002/rsa.20168},
      review={\MR{2337396 (2008e:05124)}},
}

\bib{bremaud1981}{book}{
      author={Br{\'e}maud, Pierre},
       title={Point processes and queues},
   publisher={Springer-Verlag},
     address={New York},
        date={1981},
        ISBN={0-387-90536-7},
        note={Martingale dynamics, Springer Series in Statistics},
      review={\MR{636252 (82m:60058)}},
}

\bib{peres-ding}{article}{
      author={Ding, Jian},
      author={Lubetzky, Eyal},
      author={Peres, Yuval},
       title={Mixing time of near-critical random graphs},
        date={2012},
        ISSN={0091-1798},
     journal={Ann. Probab.},
      volume={40},
      number={3},
       pages={979\ndash 1008},
         url={http://dx.doi.org.libproxy.lib.unc.edu/10.1214/11-AOP647},
      review={\MR{2962084}},
}

\bib{er-2}{article}{
      author={Erd{\H{o}}s, P.},
      author={R{\'e}nyi, A.},
       title={On the evolution of random graphs},
        date={1960},
     journal={Magyar Tud. Akad. Mat. Kutat\'o Int. K\"ozl.},
      volume={5},
       pages={17\ndash 61},
      review={\MR{0125031 (23 \#A2338)}},
}

\bib{er-1}{article}{
      author={Erd{\H{o}}s, P.},
      author={R{\'e}nyi, A.},
       title={On the evolution of random graphs},
        date={1961},
     journal={Bull. Inst. Internat. Statist.},
      volume={38},
       pages={343\ndash 347},
      review={\MR{0148055 (26 \#5564)}},
}

\bib{janson1994birth}{article}{
      author={Janson, S.},
      author={Knuth, D.},
      author={Luczak, T.},
      author={Pittel, B.},
       title={{The birth of the giant component, with an introduction by the
  editors}},
        date={1994},
     journal={Random Struct. Alg},
      volume={4},
       pages={231\ndash 358},
}

\bib{janson2010phase}{article}{
      author={Janson, S.},
      author={Spencer, J.},
       title={{Phase Transitions for Modified Erdos--R{\'e}nyi Processes}},
        date={2010},
     journal={Arxiv preprint arXiv:1005.4494},
}

\bib{janson-subcrit}{article}{
      author={Janson, Svante},
       title={The largest component in a subcritical random graph with a power
  law degree distribution},
        date={2008},
        ISSN={1050-5164},
     journal={Ann. Appl. Probab.},
      volume={18},
      number={4},
       pages={1651\ndash 1668},
         url={http://dx.doi.org.libproxy.lib.unc.edu/10.1214/07-AAP490},
      review={\MR{2434185 (2009d:05227)}},
}

\bib{janson-luczak-subcrit}{article}{
      author={Janson, Svante},
      author={Luczak, Malwina~J.},
       title={Susceptibility in subcritical random graphs},
        date={2008},
        ISSN={0022-2488},
     journal={J. Math. Phys.},
      volume={49},
      number={12},
       pages={125207, 23},
         url={http://dx.doi.org.libproxy.lib.unc.edu/10.1063/1.2982848},
      review={\MR{2484338 (2010b:05152)}},
}

\bib{bf-spencer-perkins-kang}{article}{
      author={Kang, M.},
      author={},
      author={Perkins, W},
      author={Spencer, J.},
       title={{The Bohman-Frieze Process Near Criticality}},
        date={2010},
     journal={Arxiv preprint arXiv:1106.0484v1},
}

\bib{poisson-cloning}{incollection}{
      author={Kim, Jeong~Han},
       title={Poisson cloning model for random graphs},
        date={2006},
   booktitle={International {C}ongress of {M}athematicians. {V}ol. {III}},
   publisher={Eur. Math. Soc., Z\"urich},
       pages={873\ndash 897},
      review={\MR{2275710 (2007j:05196)}},
}

\bib{liptser-mart-book}{book}{
      author={Liptser, R.~Sh.},
      author={Shiryayev, A.~N.},
       title={Theory of martingales},
      series={Mathematics and its Applications (Soviet Series)},
   publisher={Kluwer Academic Publishers Group},
     address={Dordrecht},
        date={1989},
      volume={49},
        ISBN={0-7923-0395-4},
         url={http://dx.doi.org/10.1007/978-94-009-2438-3},
        note={Translated from the Russian by K. Dzjaparidze [Kacha
  Dzhaparidze]},
      review={\MR{1022664 (90j:60046)}},
}

\bib{norris-mc}{book}{
      author={Norris, J.~R.},
       title={Markov chains},
      series={Cambridge Series in Statistical and Probabilistic Mathematics},
   publisher={Cambridge University Press},
     address={Cambridge},
        date={1998},
      volume={2},
        ISBN={0-521-48181-3},
        note={Reprint of 1997 original},
      review={\MR{1600720 (99c:60144)}},
}

\bib{pittel-subcrit}{article}{
      author={Pittel, B.~G.},
       title={On the largest component of a random graph with a subpower-law
  degree sequence in a subcritical phase},
        date={2008},
        ISSN={1050-5164},
     journal={Ann. Appl. Probab.},
      volume={18},
      number={4},
       pages={1636\ndash 1650},
         url={http://dx.doi.org.libproxy.lib.unc.edu/10.1214/07-AAP493},
      review={\MR{2434184 (2009e:60020)}},
}

\bib{riordan2011achlioptas}{article}{
      author={Riordan, O.},
      author={Warnke, L.},
       title={{Achlioptas process phase transitions are continuous}},
        date={2011},
     journal={Arxiv preprint arXiv:1102.5306},
}

\bib{riordan2012evolution}{article}{
      author={Riordan, O.},
      author={Warnke, L.},
       title={The evolution of subcritical {A}chlioptas processes},
        date={2012},
     journal={arXiv preprint arXiv:1204.5068},
}

\bib{spencer2007birth}{article}{
      author={Spencer, J.},
      author={Wormald, N.},
       title={{Birth control for giants}},
        date={2007},
     journal={Combinatorica},
      volume={27},
      number={5},
       pages={587\ndash 628},
}

\bib{van2012hypercube}{article}{
      author={van~der Hofstad, R.},
      author={Nachmias, A.},
       title={Hypercube percolation},
        date={2012},
     journal={arXiv preprint arXiv:1201.3953},
}

\end{biblist}
\end{bibdiv}


\end{document}